\RequirePackage{ifthen}
\newboolean{MPA}
\setboolean{MPA}{false}

\ifthenelse {\boolean{MPA}}
{
\documentclass{svjour3}
\smartqed
\usepackage[margin=1.3in]{geometry}
} {
\documentclass[10pt,letterpaper]{article}
\usepackage[margin=1.1in]{geometry}
}

\usepackage{amsmath}
\usepackage{amssymb}
\usepackage{hyperref}

\usepackage{mathrsfs} 

\usepackage[makeroom]{cancel} 

\usepackage{enumitem} 

\newcommand{\R}{\mathbb R}
\newcommand{\Z}{\mathbb Z}
\newcommand{\Q}{\mathbb Q}

\newcommand{\dir}{v}

\newcommand{\deno}{r}

\let\st\relax
\DeclareMathOperator{\st}{s.t.}

\DeclareMathOperator{\size}{size}
\DeclareMathOperator{\width}{width}

\DeclareMathOperator{\off}{off}

\renewcommand{\P}{\mathcal P}
\newcommand{\B}{\mathcal B}
\newcommand{\C}{\mathcal C}
\renewcommand{\S}{\mathcal S}
\newcommand{\E}{\mathcal E}

\newcommand{\floor}[1]{\lfloor#1\rfloor}

\newcommand{\ceil}[1]{\lceil#1\rceil}
\newcommand{\ceilL}[1]{\left\lceil#1\right\rceil}

\newcommand{\abs}[1]{\lvert#1\rvert}
\newcommand{\absL}[1]{\left\lvert#1\right\rvert}

\newcommand{\norm}[1]{\lVert#1\rVert_2}

\newcommand{\pnorm}[1]{\lVert#1\rVert}

\newcommand{\Fnorm}[1]{\lVert#1\rVert_F}

\newcommand{\pare}[1]{\left(#1\right)}
\newcommand{\bra}[1]{\left\{#1\right\}}
\newcommand{\sbra}[1]{\left[#1\right]}

\newcommand{\transp}{\mathsf T}

\newcommand{\constlen}{2^{p(p-1)/4}}

\newcommand{\linearized}{g}

\usepackage{comment}
\specialcomment{knownproof}{\color{blue}}{\color{black}}
\excludecomment{knownproof}

\newcounter{step}

\ifthenelse {\boolean{MPA}}
{

\newenvironment{prf}[1][]
{\begin{proof}}
{\qed \end{proof}}

\newcounter{claim} 
\renewenvironment{claim}[1][]
{\refstepcounter{claim} \begin{trivlist} \item[] {\bf Claim~\theclaim}\space#1 \itshape}
{\end{trivlist}}

\journalname{Mathematical Programming A}

} {

\usepackage{amsthm}
\newtheorem{theorem}{Theorem}
\newtheorem{proposition}{Proposition}
\newtheorem{lemma}{Lemma}

\newtheorem{corollary}{Corollary}

\newtheorem{example}{Example}

\newenvironment{prf}[1][]
{\begin{proof}}
{\end{proof}}

}

\newcommand{\objS}{f^{\mathrm{S}}}
\newcommand{\objSinf}{f^{\mathrm{S}}_{\inf}}
\newcommand{\objSsup}{f^{\mathrm{S}}_{\sup}}

\newcommand{\low}{\ell}
\newcommand{\upp}{u}

\begin{document}

\title{Rational Jacobi Rotations and the Complexity of Approximating Mixed Integer Quadratic Programming}

\ifthenelse {\boolean{MPA}}
{
\titlerunning{Rational Jacobi Rotations and the Complexity of Approximating MIQP}
\authorrunning{A. Del Pia}

\author{Alberto Del Pia}
\institute{Alberto~Del~Pia \at
              Department of Industrial and Systems Engineering 
              \& Wisconsin Institute for Discovery \\
              University of Wisconsin-Madison, Madison, WI, USA \\
              \email{delpia@wisc.edu}}
}
{
\author{Alberto Del Pia
\thanks{Department of Industrial and Systems Engineering \& Wisconsin Institute for Discovery,
             University of Wisconsin-Madison, Madison, WI, USA.
             E-mail: {\tt delpia@wisc.edu}.}}
}

\date{July 31, 2026}


\maketitle

\begin{abstract}
We present an algorithm that finds an $\epsilon$-approximate solution to a mixed integer quadratic programming (MIQP) problem, and that runs on a Turing machine in time polynomial in the size of the instance and in $1/\epsilon$, provided that the number of integer variables and the number of negative eigenvalues of the Hessian of the objective function are fixed.
Unless $\mathrm{P}=\mathrm{NP}$, both restrictions are necessary, so this completes the characterization of the complexity of approximating MIQP in terms of the number of integer variables and the inertia of the Hessian; the result is new already in the purely continuous case.
The main ingredient is a polynomial-time simultaneous diagonalization algorithm: it computes a rational change of basis that maps a given ellipsoid, presented in factored form, exactly to a ball, while making the objective function separable up to an arbitrarily small perturbation and preserving the inertia of its Hessian.
The classical construction, in which the objective function is made exactly separable, requires a change of basis that is in general irrational, and cannot be carried out on a Turing machine; ours rests instead on rational Jacobi rotations, which we believe to be of independent interest.

\ifthenelse {\boolean{MPA}}
{
} {}
\end{abstract}

\ifthenelse {\boolean{MPA}}
{}{
}

\section{Introduction}

A mixed integer quadratic programming problem is an optimization problem of the form
\begin{align}
\label{prob MIQP}
\tag{MIQP}
\begin{split}
\min & \quad x^\transp H x + h^\transp x \\
\st & \quad Wx \le w \\
& \quad x \in \Z^p \times \R^{n-p}.
\end{split}
\end{align} 
In this formulation, $x \in \R^n$ is the vector of variables and the input data consists of a symmetric $H \in \Q^{n \times n}$, which we refer to as the Hessian of the objective function, $h \in \Q^n$, $W \in \Q^{m \times n}$, $w \in \Q^m$, and $p \in \{0,1,\dots,n\}$, the number of integer variables.

Already deciding whether Problem~\ref{prob MIQP} is feasible is NP-hard, since it contains integer linear feasibility as a special case.
Lenstra~\cite{Len83} showed that this can be done in polynomial time provided that $p$ is fixed, and we place ourselves in this setting throughout.

The difficulty of Problem~\ref{prob MIQP} is governed, besides $p$, by the spectrum of the matrix $H$.
The \emph{inertia} of a symmetric matrix is an ordered triple $(k_-,k_0,k_+)$, where $k_-$ is the number of negative eigenvalues, $k_0$ the number of zero eigenvalues, and $k_+$ the number of positive eigenvalues.
If $H$ is positive semidefinite, that is, if $k_-=0$, then Problem~\ref{prob MIQP} can be solved in polynomial time on a Turing machine, for fixed $p$~\cite{dP25SIOPT}.
As soon as $H$ has a single negative eigenvalue, however, the problem becomes NP-hard, already for $p=0$~\cite{ParVav91}.
In the non-convex case one can therefore only aim at approximation algorithms, and this is the object of this paper.

To state our main result, we first define the concepts of $\epsilon$-approximate solution and of size.
Consider an instance of a general optimization problem $\inf \bra{ f(x) : x \in \S}$, with $\S \neq \emptyset$.
For $\epsilon \in [0,1]$, we say that $x^\diamond \in \S$ is an \emph{$\epsilon$-approximate solution} if
\begin{equation*}
f(x^\diamond) - f_{\inf} \le \epsilon \cdot (f_{\sup} - f_{\inf}),
\end{equation*}
where $f_{\inf} := \inf_{x \in \S} f(x)$ and $f_{\sup} := \sup_{x \in \S} f(x)$.
Note that $0$-approximate solutions coincide with optimal solutions and $1$-approximate solutions coincide with feasible solutions.
For any $\epsilon > 0$, if $f_{\inf} = - \infty$, or $f_{\sup} = + \infty$, then our definition loses its value because any feasible point is an $\epsilon$-approximate solution.
This definition of approximation has some well-known invariance properties which make it a natural choice for the problems considered in this paper.
For instance, it is unchanged when the objective function is multiplied by a positive scalar or shifted by a constant, the feasible region being left untouched, and it is also unchanged when an affine transformation is applied simultaneously to the objective function and to the feasible region, like for example a change of basis.
This definition has been used in our earlier work~\cite{dP16IPCO,dP18MPB,dP23bMPA}, and in numerous other places, including~\cite{NemYud83,Vav92c,Vav92i,BelRog95,KleLauPar06}.

We define the \emph{sizes} of a rational number $p/q$ (where $p$ and $q$ are relatively prime integers), of a rational vector $v \in \Q^n$ and of a rational matrix $A \in \Q^{m \times n}$ as in \cite{SchBookIP}:
\begin{align*}
\size(p/q) & := 1+\ceil{\log(|p|+1)}+\ceil{\log(|q|+1)}, \\
\size(v) & := n + \sum_{i=1}^n \size(v_i), \\
\size(A) & := mn + \sum_{i=1}^m \sum_{j=1}^n \size(a_{ij}).
\end{align*}
Throughout the paper, $\log$ denotes the base-$2$ logarithm, and $\ln$ denotes the natural logarithm.
Given an instance of Problem~\ref{prob MIQP} (i.e., specific data $H,h,W,w$), we define its \emph{size} as $\size(H) + \size(h) + \size(W) + \size(w)$.
We are now ready to state the main result of this paper.

\begin{theorem}[Approximation algorithm for Problem~\ref{prob MIQP}]
\label{th MIQP algorithm}
Assume that the objective function of Problem~\ref{prob MIQP} is bounded below on the feasible region, and let $\epsilon\in(0,1]$ be rational.
There is an algorithm that either detects that Problem~\ref{prob MIQP} is infeasible, or finds an $\epsilon$-approximate solution to Problem~\ref{prob MIQP}.
The algorithm runs on a Turing machine in time polynomial in the size of the instance, in $\size(\epsilon)$, and in $1/\epsilon$, provided that $p$ and the number $k_-$ of negative eigenvalues of $H$ are fixed.
\end{theorem}

The two parameters that Theorem~\ref{th MIQP algorithm} keeps fixed are precisely the two that must be fixed, unless $\mathrm{P}=\mathrm{NP}$.
If the assumption that $p$ is fixed is dropped, the algorithm would decide in polynomial time whether the feasible region of Problem~\ref{prob MIQP} is empty, which is NP-hard.
If the assumption that $k_-$ is fixed is dropped, the algorithm could be used to solve 3SAT in polynomial time, already for $p=0$~\cite{Vav92c}.
In this sense Theorem~\ref{th MIQP algorithm} completes the characterization of the complexity of approximating Problem~\ref{prob MIQP} in terms of $p$ and of the inertia of $H$.
The dependence of the running time on $1/\epsilon$ cannot be dropped either: an algorithm running in time polynomial in the size of the instance and in $\size(\epsilon)$ alone could be used to solve quadratic programming with one negative eigenvalue in polynomial time~\cite{Vav92c}, and the latter problem is NP-hard~\cite{ParVav91}.

The assumption that the objective function is bounded below on the feasible region is imposed rather than checked, and necessarily so: deciding whether it holds is NP-hard, even if $p=0$ and $H$ has only two negative eigenvalues.
This follows from Theorem~2 in~\cite{dP23bMPA}, which is stated in terms of the rank of $H$; the instances constructed in its proof have a Hessian with exactly two negative eigenvalues and one positive eigenvalue.

\medskip
\noindent
\textbf{Related work. \ }
Theorem~\ref{th MIQP algorithm} can be seen as the conclusion of a line of approximation algorithms with theoretical guarantees for quadratic programming (Problem~\ref{prob MIQP} with $p=0$) and for mixed integer quadratic programming.
For quadratic programming, Vavasis gave an approximation algorithm for the concave case ($k_+=0$) with $k_-$ fixed~\cite{Vav92c}, and then for the more general case with only $k_-$ fixed~\cite{Vav92i}.
In the mixed integer setting (with $p$ fixed), the concave case ($k_+=0$) with $k_-$ fixed was settled in~\cite{dP16IPCO,dP18MPB}, and the fixed rank case, with both $k_-$ and $k_+$ fixed, in~\cite{dP23bMPA}.
Theorem~\ref{th MIQP algorithm} subsumes all of these: it requires only that $k_-$ is fixed, and in particular it places no restriction whatsoever on the number $k_+$ of positive eigenvalues, hence none on the rank of $H$.
All the algorithms mentioned above run in polynomial time on a Turing machine, with one exception: the algorithm of~\cite{Vav92i} relies on a simultaneous diagonalization of two matrices which, as we explain below, cannot be carried out on a Turing machine, since the matrix it produces is in general irrational.
We stress that our result is thus new already for $p=0$: for quadratic programming with a fixed number of negative eigenvalues, no polynomial-time approximation algorithm in the Turing model of computation was known.

All the algorithms above, including ours, are based on the classical technique of mesh partitioning and linear underestimators, in which some of the quadratic terms of the objective function are replaced, on each piece of a partition of the feasible region, by linear underestimators; this idea has been used in optimization since at least the 1950s.
For this technique to yield an $\epsilon$-approximate solution, one first needs a change of variables that makes the objective function \emph{separable}, that is, that replaces $H$ by a diagonal matrix.
In the concave case, this change of variables is the only preprocessing that is needed, and it can be achieved on a Turing machine in polynomial time by the symmetric decomposition algorithm of Dax and Kaniel~\cite{DaxKan77}, whose running time was analyzed in~\cite{dP23bMPA}: that algorithm produces a rational nonsingular $L$ with $L^\transp H L$ diagonal, and, by Sylvester's law of inertia, the inertia is automatically preserved.

In the indefinite case the situation is considerably more delicate, because a separable objective function is no longer enough: one also needs the polyhedron described by the linear constraints, which we may assume to be bounded, to be of \emph{spherical form}, that is, sandwiched between two concentric balls whose radii differ by a controlled factor.
The reason is that the analysis needs to know that the polytope is neither too large nor too small in any direction, and this is exactly what the two balls guarantee.
If the rank of $H$ is fixed, which is the setting of~\cite{dP23bMPA}, separability and spherical form are only needed on a subspace of fixed dimension, and the decomposition of Dax and Kaniel essentially suffices.
If the rank of $H$ is not fixed, as in~\cite{Vav92i} and in the present paper, this simplification is not available, and both properties are needed in full.
The two requirements pull in opposite directions: a general nonsingular change of variables can diagonalize $H$, but it distorts the feasible region and destroys sphericity, whereas an orthogonal change of variables preserves balls but is, in general, irrational.

\medskip
\noindent
\textbf{The main obstacle. \ }
In~\cite{Vav92i} both requirements are met at once by constructing, given the symmetric matrix $H$ and a symmetric positive definite matrix $M$ describing an ellipsoid that approximates the feasible region, a nonsingular matrix $L$ such that
\begin{align}
\label{eq classical simultaneous}
L^\transp H L \text{ is diagonal} \qquad \text{and} \qquad L^\transp M L = I,
\end{align}
where $I$ denotes the identity matrix; the second condition is what turns the ellipsoid into a ball.
Such a matrix $L$ always exists over the reals, and is classically obtained from a Cholesky factorization $M = C^\transp C$ followed by a spectral decomposition (see, e.g., Algorithm~8.7.1 in~\cite{GolVan4th}), but both steps are in general irrational, and therefore this construction cannot be carried out on a Turing machine.
Throughout, we assume that $M$ is given already in factored form, that is, that a rational nonsingular $C$ with $M = C^\transp C$ is part of the input; this is how $M$ arises in our application, where $C$ is produced by Lenstra's ellipsoid rounding algorithm, and it is a genuine restriction, since such a $C$ need not exist (Example~\ref{ex M not factored}).
The irrationality of the spectral decomposition, however, remains, and is already apparent in dimension two.

\begin{example}
\label{ex no rational L}
Let $n := 2$, let $M := I$, which is in factored form with $C := I$, and let $H := \begin{bmatrix} 0 & 1 \\ 1 & 1 \end{bmatrix}$.
Suppose that $L$ satisfies \eqref{eq classical simultaneous}.
Since $M = I$, the second condition reads $L^\transp L = I$, so $L$ is orthogonal and $L^\transp H L = L^{-1} H L$ is similar to $H$; being diagonal, it has the eigenvalues $\frac{1 \pm \sqrt5}{2}$ of $H$ on its diagonal, and these are irrational.
Hence $L^\transp H L$, and therefore $L$, cannot be rational, even though $M$ is given in factored form.
\end{example}

The obstruction is not an artifact of this small example: the columns of $L$ are generalized eigenvectors of $H$ with respect to $M$, that is, nonzero vectors $v$ with $Hv = \lambda Mv$ for some $\lambda \in \R$, so computing them exactly amounts to computing eigenvalues exactly, which for $n \ge 5$ is not in general possible by radicals.

\medskip
\noindent
\textbf{Our approach. \ }
Our approach keeps the second condition in \eqref{eq classical simultaneous} and relaxes the first.
We show (Theorem~\ref{th two matrices}) that a rational nonsingular $L$ satisfying $L^\transp M L = I$ \emph{exactly}, and for which $L^\transp H L$ is diagonal \emph{up to a perturbation $E$ of prescribed Frobenius norm at most $\delta$}, can be computed in polynomial time; moreover the diagonal part $D$ has the same inertia as $H$, which is what allows the number $k_-$ of negative eigenvalues to be tracked through the reduction.
Relaxing exact diagonality is precisely what makes $L$ rational.
The transformed problem is obtained by applying the change of variables $x = Ly$ (up to a translation), which turns the quadratic part of the objective function into $y^\transp (D+E) y$, and then discarding $E$ so as to be left with a separable objective function.
The price to pay is that this problem is no longer equivalent to the original one, since discarding $E$ perturbs its objective function.
We show that this is harmless, in the sense that an approximate solution to the transformed problem yields an approximate solution to Problem~\ref{prob MIQP}, provided that the objective function varies enough on the feasible region.
The transformed problem is in spherical form, so mesh partitioning and linear underestimators apply, and the resulting subproblems are mixed integer convex quadratic programs, which we solve with the polynomial-time algorithm of~\cite{dP25SIOPT}.
A theorem of the alternative, based on a flatness theorem for mixed integer lattices, removes the proviso: it either produces an approximate solution, certifying that the objective function does vary enough, or splits the instance into sub-instances with one fewer integer variable, on which we recurse.


The engine behind Theorem~\ref{th two matrices} is a rational analogue of the classical Jacobi eigenvalue algorithm, which we develop in Section~\ref{sec Jacobi} and which we believe is of independent interest.
Its goal is Theorem~\ref{th rational Jacobi near diag}: given a symmetric matrix $H$ and a tolerance $\delta \in (0,1]$, compute in polynomial time a rational \emph{orthogonal} change of basis after which the off-diagonal part of $H$ has Frobenius norm at most $\delta$.
A Jacobi rotation acts on a symmetric matrix by an orthogonal similarity that annihilates one off-diagonal entry; iterating on the largest off-diagonal entry drives the matrix towards diagonal form at a geometric rate.
The rotation that annihilates a given entry involves square roots and is in general irrational, so the classical algorithm lives in the real RAM model.
We observe that the rational points are dense on the unit circle, via the Weierstrass substitution
$$
c = \frac{1-u^2}{1+u^2}, \qquad s = \frac{2u}{1+u^2},
$$
and we use it to construct, in polynomial time, \emph{rational} Jacobi rotations: matrices that are exactly orthogonal, have rational entries of controlled size, and annihilate the chosen off-diagonal entry not exactly but up to any prescribed tolerance.
Exact orthogonality is what makes this work, and it is what distinguishes our rotations from simply rounding the entries of a real one: it keeps the Frobenius norm constant along the iterations, so that annihilating an off-diagonal entry decreases the off-diagonal norm, and it makes the accumulated change of basis orthogonal.

Two further ingredients turn near-diagonalization into diagonalization with a guarantee on the inertia.
First, a classical bound of Cauchy on the roots of a polynomial shows that the nonzero eigenvalues of an integer symmetric matrix are bounded away from zero by a quantity that depends only on the size of the matrix.
Second, a perturbation argument then shows that, once the off-diagonal part has been driven below that threshold, the small entries can safely be rounded to zero without changing the inertia.
Together with the rational Jacobi rotations, this yields Theorem~\ref{th diagonalize}, which, by a rational orthogonal change of basis, makes $H$ diagonal up to a perturbation of Frobenius norm at most $\delta$, and guarantees that the diagonal part has the same inertia as $H$.
Applying it to $C^{-\transp} H C^{-1}$, and composing the two changes of basis, yields Theorem~\ref{th two matrices}.
Finally, if in place of the matrix $M$ we are given a full-dimensional rational polytope, and $C$ is obtained by applying the ellipsoid rounding algorithm of Lenstra~\cite{Len83} to it, we obtain Theorem~\ref{th simultaneous diagonalization}, where the change of basis in addition maps the polytope to one sandwiched between two concentric balls, as required for spherical form.

\medskip
\noindent
\textbf{Organization of the paper. \ }
Section~\ref{sec Jacobi} develops exact and rational Jacobi rotations, and culminates in the near-diagonalization algorithm of Theorem~\ref{th rational Jacobi near diag}.
Section~\ref{sec diagonalization} turns it into the inertia-preserving diagonalization algorithm of Theorem~\ref{th diagonalize}, and then into the simultaneous diagonalization algorithms of Theorems~\ref{th two matrices} and~\ref{th simultaneous diagonalization}.
Section~\ref{sec MIQP} introduces spherical form MIQP, shows how to reduce Problem~\ref{prob MIQP} to it, solves it by mesh partitioning and linear underestimators, proves two theorems of the alternative, and assembles the proof of Theorem~\ref{th MIQP algorithm}.












\section{Jacobi rotations}
\label{sec Jacobi}

We say that $(c, s) \in \R^2$ is a \emph{cosine-sine pair} associated with some angle $\theta \in \R$ if $c = \cos(\theta)$ and $s = \sin(\theta)$.
Note that $(c, s)$ is a cosine-sine pair if and only if $c^2+s^2=1$.
Throughout this section, $n$ denotes an integer with $n \ge 2$.
Given an index pair $(p,q)$ that satisfies $1 \le p < q \le n$ and a cosine-sine pair $(c, s)$ associated with some angle $\theta$, a \emph{Jacobi rotation}, denoted by $J(p,q; c,s)$, is an $n \times n$ orthogonal matrix that differs from the $n \times n$ identity matrix only in the $2 \times 2$ principal submatrix corresponding to row/columns $p,q$, which has the form
$$
\begin{bmatrix}
c & s \\
-s & c
\end{bmatrix}.
$$
Jacobi rotations are also known as \emph{Givens rotations.}
Multiplying a vector in $\R^n$ on the left by $J(p, q; c,s)$ rotates the vector through $\theta$ radians clockwise in the $(p,q)$ plane.

Before proceeding, we first recall some standard matrix norms. Given a matrix $A \in \R^{m \times n}$, its \emph{spectral norm} and \emph{Frobenius norm} are defined by
\begin{align*} 
\norm{A} := \sup_{x \neq 0} \frac{\norm{Ax}}{\norm{x}}, \qquad
\Fnorm{A} := \sqrt{\sum_{i=1}^m \sum_{j=1}^n a_{ij}^2}.
\end{align*}
It is well known that $\norm{A} \le \Fnorm{A}$ for every $A \in \R^{m\times n}$.
Furthermore, viewing a vector $v \in \R^n$ as an $n \times 1$ matrix, its spectral norm $\norm{v}$ coincides with its Euclidean norm.
For a matrix $A \in \R^{n\times n}$ we use the notation
$$
\off(A):=\sqrt{\sum_{i=1}^n \sum_{\substack{j=1 \\ j \neq i}}^n a_{i j}^2}
$$
for the off-diagonal analogue of the Frobenius norm.
Note that $\off(A)\le\Fnorm{A}$ for every $A\in\R^{n\times n}$, since $\Fnorm{A}^2=\off(A)^2+\sum_{i=1}^n a_{ii}^2$.

\subsection{The effect of a single rotation}
\label{sec Jacobi norm}

Both the exact rotations of Section~\ref{sec Jacobi exact} and the rational ones of Section~\ref{sec rational Jacobi} act on $\off(\cdot)$ through the same identity, which we record here once and for all.
It holds for an arbitrary cosine-sine pair: in the exact case the rotation annihilates the $(p,q)$ entry of the transformed matrix, while in the rational case that entry is merely small.
In the former case the identity is well known.

Note that, given a symmetric matrix $A \in \R^{n\times n}$, an index pair $(p,q)$ with $1\le p<q\le n$, a cosine-sine pair $(c,s)$, and $J:=J(p,q;c,s)$, the entries in rows and columns $p,q$ of $B:=J^\transp A J$ satisfy
\begin{align}
\label{eq pq}
\begin{bmatrix}
b_{p p} & b_{p q} \\
b_{q p} & b_{q q}
\end{bmatrix}
=
\begin{bmatrix}
c & s \\
-s & c
\end{bmatrix}^\transp
\begin{bmatrix}
a_{p p} & a_{p q} \\
a_{q p} & a_{q q}
\end{bmatrix}
\begin{bmatrix}
c & s \\
-s & c
\end{bmatrix}.
\end{align}

\begin{lemma}
\label{lem Jacobi norm}
Let $A \in \R^{n \times n}$ be symmetric and let $(p, q)$ be an index pair that satisfies $1 \le p<q \le n$.
Let $(c,s)$ be a cosine-sine pair, let $J:=J(p,q;c,s)$, and let $B:=J^\transp A J$.
Then
\begin{align*}
\off(B)^2 = \off(A)^2-2 a_{p q}^2 + 2 b_{p q}^2.
\end{align*}
\end{lemma}

\begin{prf}
Since $A$ is symmetric, so is $B=J^\transp A J$; and since $J$ is orthogonal and the Frobenius norm is preserved by orthogonal transformations, $\Fnorm{B}^2=\Fnorm{A}^2$.
Since also the $2\times2$ block of $J$ given by $\begin{bmatrix}c&s\\-s&c\end{bmatrix}$ is orthogonal, from \eqref{eq pq} we obtain
$$
a_{p p}^2+a_{q q}^2+2 a_{p q}^2=b_{p p}^2+b_{q q}^2+2 b_{p q}^2.
$$
Since $B$ agrees with $A$ except in rows and columns $p$ and $q$ (so that $b_{ii}=a_{ii}$ for every $i\neq p,q$), we have
$$
\sum_{i=1}^n b_{ii}^2 = \sum_{i\neq p,q} b_{ii}^2 + b_{pp}^2+b_{qq}^2 
= \sum_{i\neq p,q} a_{ii}^2 + b_{pp}^2+b_{qq}^2 
= \sum_{i=1}^n a_{ii}^2 + 2 a_{p q}^2 - 2 b_{p q}^2.
$$
Hence,
\begin{align*}
\off(B)^2 & =\Fnorm{B}^2 - \sum_{i=1}^n b_{i i}^2 
= \Fnorm{A}^2 - \sum_{i=1}^n a_{i i}^2 - 2 a_{p q}^2 + 2 b_{p q}^2 
=\off(A)^2-2 a_{p q}^2 + 2 b_{p q}^2. 
\end{align*}
\end{prf}

\subsection{Exact Jacobi rotations}
\label{sec Jacobi exact}

We first treat the classical case, in which each rotation zeroes out the chosen off-diagonal entry exactly.
The results of this subsection are well known; we include their proofs for completeness, and because the rational versions of Section~\ref{sec rational Jacobi} follow the same pattern.

\subsubsection{Construction}

Off-diagonal entries of a symmetric matrix $A \in \R^{n \times n}$ can be zeroed out by replacing $A$ with $J^\transp A J$, for some suitable Jacobi rotation $J$.
The following lemma makes this construction explicit, giving closed-form formulas for a Jacobi rotation that zeroes out a prescribed off-diagonal entry.

\begin{lemma}
\label{lem exact Jacobi construction}
Let $A \in \R^{n \times n}$ be symmetric and let $(p, q)$ be an index pair that satisfies $1 \le p<q \le n$.
If $a_{p q} = 0$, let $t := 0$; otherwise, let
$$
\tau:=\frac{a_{q q}-a_{p p}}{2 a_{p q}}, \qquad 
t := 
\begin{cases} 
\sqrt{1+\tau^2}-\tau \in (0,1] & \text{if } \tau \ge 0 \\ 
-\tau-\sqrt{1+\tau^2} \in (-1,0) & \text{if } \tau<0.
\end{cases}
$$
In both cases, let
$$
\theta:=\arctan(t), 
\qquad
c:= \frac{1}{\sqrt{1+t^2}}, 
\qquad 
s:=\frac{t}{\sqrt{1+t^2}}.
$$
Then $(c, s)$ is a cosine-sine pair associated with angle $\theta \in (-\pi/4,\pi/4]$.
Furthermore, if we let $B:=J^\transp A J$, where $J$ denotes the Jacobi rotation $J(p, q; c,s)$, then $b_{pq} = 0$.
\end{lemma}

\begin{prf}
In all cases $t\in(-1,1]$: this is clear if $a_{pq}=0$, and otherwise it follows from $\abs{\tau}<\sqrt{1+\tau^2}\le1+\abs{\tau}$, the second inequality being strict for $\tau\neq0$, which also gives the ranges for $t$ stated in the lemma.
Since $\theta=\arctan(t)\in(-\pi/2,\pi/2)$, we have $\cos(\theta)>0$.
Since $\tan(\theta)=t$, we have
$$
\cos(\theta)=\frac1{\sqrt{1+t^2}}=c, \qquad \sin(\theta)=\frac{t}{\sqrt{1+t^2}}=s,
$$
so $(c,s)$ is the cosine-sine pair associated with $\theta$. Since $t\in(-1,1]$, and since $\arctan$ is increasing with $\arctan(-1)=-\pi/4$ and $\arctan(1)=\pi/4$, we get $\theta\in(-\pi/4,\pi/4]$.

It remains to show $b_{pq}=0$.
If $a_{pq}=0$, then $t=0$, hence $c=1$ and $s=0$, and \eqref{eq pq} gives $b_{pq}=a_{pq}(c^2-s^2)+(a_{pp}-a_{qq})cs=0$.
Assume now $a_{pq}\neq0$. From \eqref{eq pq},
$$
b_{pq} = a_{pq}(c^2-s^2)+(a_{pp}-a_{qq})cs = a_{pq}\cdot\frac{1-t^2}{1+t^2} + (a_{pp}-a_{qq})\cdot\frac{t}{1+t^2} = \frac{a_{pq}\pare{(1-t^2)-2\tau t}}{1+t^2},
$$
where we used $a_{pp}-a_{qq}=-2a_{pq}\tau$ (from the definition of $\tau$, valid since $a_{pq}\neq0$). It then suffices to show $t^2+2\tau t-1=0$.

If $\tau\ge0$, then $t=\sqrt{1+\tau^2}-\tau$, so
$$
t^2+2\tau t = \pare{1+2\tau^2-2\tau\sqrt{1+\tau^2}} + \pare{2\tau\sqrt{1+\tau^2}-2\tau^2} = 1.
$$
If $\tau<0$, then $t=-\tau-\sqrt{1+\tau^2}$, so
$$
t^2+2\tau t = \pare{1+2\tau^2+2\tau\sqrt{1+\tau^2}} + \pare{-2\tau^2-2\tau\sqrt{1+\tau^2}} = 1.
$$
In either case $t^2+2\tau t=1$, hence $b_{pq}=0$, as desired.
\end{prf}

\subsubsection{Repeated updates}

Combining Lemmas~\ref{lem Jacobi norm} and \ref{lem exact Jacobi construction}, repeatedly zeroing out the largest-magnitude off-diagonal entry, one Jacobi rotation at a time, drives $\off(\cdot)$ toward zero at a geometric rate; this is a classical fact, which we state below in the form convenient for our later use.

Let $A \in \R^{n \times n}$ be symmetric, and let $(p, q)$ be an index pair that satisfies $1 \le p<q \le n$ and $|a_{p q}| = \max_{i \neq j}|a_{ij}|$.
We say that $B \in \R^{n\times n}$ is obtained by applying a \emph{Jacobi update} to $A$ if $B = J^\transp A J$ for some cosine-sine pair $(c,s)$, where $J:=J(p,q;c,s)$, such that $b_{pq}=0$.
Note that such a pair $(c,s)$ always exists, and can be constructed via the closed-form formulas in Lemma~\ref{lem exact Jacobi construction}.

\begin{lemma}
\label{lem exact Jacobi sweep}
Let $A \in \R^{n \times n}$ be symmetric and let $\delta \in (0,1]$.
For every nonnegative integer $k$, let $A^{(k)}$ denote a matrix obtained from $A$ by $k$ successive Jacobi updates (with $A^{(0)}:=A$).
Then, there exists 
$$
k \le \max\pare{0, \ \left\lceil \frac{n(n-1)}{2} \ln\pare{\frac{\off(A)^2}{\delta^2}}\right\rceil}
$$
such that $\off(A^{(k)}) \le \delta$.
\end{lemma}

\begin{prf}
Let $N:=\frac{n(n-1)}{2}$ and $\xi := \ln\pare{\off(A)^2/\delta^2}$.
If $\off(A)^2 < \delta^2$, then $\off(A)\le\delta$, and $k=0$ satisfies the claim. We may therefore assume $\off(A)^2 \ge \delta^2$, so that $\xi\ge0$.
Let $K := \ceil{N\xi}$, a nonnegative integer. We show that $\off(A^{(K)}) \le \delta$.
Let $a_{pq}$ be the off-diagonal entry of $A$ with largest absolute value. Then
$$
\off(A)^2 \le n(n-1) a_{p q}^2 = 2 N a_{p q}^2.
$$
Since a Jacobi update satisfies $b_{pq}=0$, Lemma~\ref{lem Jacobi norm} gives
$$
\off(A^{(1)})^2 = \off(A)^2-2 a_{p q}^2 \le \pare{1-\frac{1}{N}} \off(A)^2.
$$
Since the same argument applies to each $A^{(k)}$, whose pivot is again an off-diagonal entry of largest absolute value, by induction
$$
\off(A^{(K)})^2 \le\pare{1-\frac{1}{N}}^{K} \off(A)^2.
$$
Since $K \ge N\xi$ and $0 \le 1-\frac1N<1$, we have $\pare{1-\frac1N}^K \le \pare{1-\frac1N}^{N\xi}$.
Since $1-\frac1N \le e^{-1/N}$, we also have $\pare{1-\frac1N}^{N\xi} \le e^{-\xi}$.
Thus, 
$$
\off(A^{(K)})^2 \le \pare{1-\frac1N}^K\off(A)^2 \le e^{-\xi} \off(A)^2
= \frac{\delta^2}{\off(A)^2}\off(A)^2
=\delta^2,
$$
so $\off(A^{(K)})\le\delta$, and $k=K$ satisfies the claim.
\end{prf}

\subsubsection{Near-diagonalization via exact rotations}

Combining Lemmas~\ref{lem exact Jacobi construction} and \ref{lem exact Jacobi sweep} yields a real-RAM algorithm for driving a symmetric matrix arbitrarily close to diagonal form; we record this immediate consequence below before turning, in the next section, to a version of this result suitable for computation on a Turing machine.
In the real RAM model we assume, as is standard, that square roots can be computed at unit cost.
Throughout, we state our bounds in terms of $\log(1/\delta)$ when they depend only on the precision requested, as for the number of arithmetic operations in the real RAM model, and in terms of $\size(\delta)$ when they depend on the encoding of $\delta$, as for running times on a Turing machine.
The two are not interchangeable: for every rational $\delta \in (0,1]$ we have $\ceil{\log(1/\delta)} \le \size(\delta)$, but $\size(\delta)$ also accounts for the number of bits needed to encode $\delta$, which may be much larger than $\log(1/\delta)$.
\begin{proposition}[Near-diagonalization via exact Jacobi rotations]
\label{prop exact Jacobi near diag}
Let $A \in \R^{n \times n}$ be symmetric and let $\delta \in (0,1]$.
There is an algorithm that computes an orthogonal matrix $L \in \R^{n \times n}$ such that $\off(L^\transp A L) \le \delta$.
The algorithm performs $O\pare{n^4\pare{1+\log(1/\delta)+\log(\max(1,\Fnorm{A}))}}$ arithmetic operations in the real RAM model.
\end{proposition}

\begin{prf}
From Lemma~\ref{lem exact Jacobi sweep}, there exists 
$$
k \le \max\pare{0,\ \left\lceil \frac{n(n-1)}{2} \ln\pare{\frac{\off(A)^2}{\delta^2}}\right\rceil}
$$ 
such that after $k$ Jacobi updates, the obtained matrix $A^{(k)}$ satisfies $\off(A^{(k)})\le\delta$.
Let $J_1, \dots, J_k$ be the corresponding Jacobi rotations, each constructed via Lemma~\ref{lem exact Jacobi construction} to zero out the largest-magnitude off-diagonal entry, and let $L := J_1 \cdots J_k$, so that $L^\transp A L = A^{(k)}$. As a product of orthogonal matrices, $L$ is orthogonal.

\smallskip
\noindent
\textbf{Running time.}
Since $\off(A)\le\Fnorm A$, Lemma~\ref{lem exact Jacobi sweep} bounds the number of updates by
$$
k = O\pare{n^2\pare{\log(1/\delta)+\log(\max(1,\Fnorm A))}}.
$$
Each update costs $O(n^2)$ arithmetic operations: $O(n^2)$ comparisons to find the pivot $(p,q)$, $O(1)$ to construct $J_{t+1}$ from the closed-form formulas of Lemma~\ref{lem exact Jacobi construction}, and $O(n^2)$ to form $A^{(t+1)} = J_{t+1}^\transp A^{(t)} J_{t+1}$, since only rows and columns $p,q$ change. Taking into account the $O(n^2)$ arithmetic operations needed to form the output $L$, which are performed even when $k=0$, the total is therefore
\[
O\pare{n^2(1+k)} = O\pare{n^4\pare{1+\log(1/\delta)+\log(\max(1,\Fnorm A))}}. 
\]
\end{prf}

\subsection{Rational Jacobi rotations}
\label{sec rational Jacobi}

Proposition~\ref{prop exact Jacobi near diag} drives $\off(\cdot)$ below any prescribed tolerance using a number of arithmetic operations that is polynomial in $n$ and in $\log(1/\delta)$, but it does so in the real RAM model, and it cannot be implemented on a Turing machine: both the Jacobi rotation $J$ constructed in Lemma~\ref{lem exact Jacobi construction} and the resulting matrix $B:=J^\transp A J$ can have irrational entries, as suggested by the presence of square roots in the formulas defining $c$ and $s$.

\begin{example}
\label{ex irrational rotation}
Let $n := 2$, let $A := \begin{bmatrix} 0 & 1 \\ 1 & 1 \end{bmatrix}$, and let $(p,q) := (1,2)$.
We have $\tau = 1/2$ and $t = (\sqrt5-1)/2$ in Lemma~\ref{lem exact Jacobi construction},
so that
$$
c = \sqrt{\frac{5+\sqrt5}{10}}, \qquad s = \sqrt{\frac{5-\sqrt5}{10}}
$$
are both irrational (since $c^2=\frac{5+\sqrt5}{10}$ and $s^2=\frac{5-\sqrt5}{10}$ are themselves irrational, as $\sqrt5\notin\Q$).
Thus, even though $A$ has integer entries, the Jacobi rotation $J(p,q;c,s)$
diagonalizing $A$ has irrational entries. Moreover, $B:=J^\transp A J$ has irrational entries as well:
$$
B
=
J^\transp A J
=
\begin{bmatrix} c & s \\ -s & c \end{bmatrix}^\transp
\begin{bmatrix} 0 & 1 \\ 1 & 1 \end{bmatrix}
\begin{bmatrix} c & s \\ -s & c \end{bmatrix}
=
\begin{bmatrix} \dfrac{1-\sqrt5}{2} & 0 \\[1mm] 0 & \dfrac{1+\sqrt5}{2} \end{bmatrix}.
$$
\end{example}

In practice, $c$ and $s$ are often rounded, but this in general destroys orthogonality, so the resulting matrix is not a Jacobi rotation.
Our next goal is instead to approximate a target Jacobi rotation by another Jacobi rotation that is rational and can be computed in polynomial time.
This is possible because rational points are dense on the two-dimensional unit circle, as can be seen via the Weierstrass substitution
$$
c = \frac{1-u^2}{1+u^2}, \qquad s = \frac{2u}{1+u^2},
$$
also known as the $\tan(\theta/2)$ substitution: as $u$ ranges over the rationals, this parametrizes a dense set of rational points on the circle.
We will use this substitution throughout our proofs, together with the following simple lemma.

\begin{lemma}
\label{lem approx bounds}
Let $\delta \ge 0$ and let $a, \hat a, b, \hat b \in \R$ be such that $\abs{a-\hat a} \le \delta$ and $\abs{b-\hat b} \le \delta$.
\begin{enumerate}
\item
\label{lem approx bounds square}
If $\abs{a}, \abs{\hat a} \le 1$, then
$\abs{a^2-\hat a^2} \le 2 \delta$.
\item
\label{lem approx bounds prod}
If $\abs{a}, \abs{\hat b} \le 1$, then
$\abs{ab-\hat a \hat b} \le 2 \delta$.
\item
\label{lem approx bounds ratio}
If $\abs{b},\abs{\hat b} \ge 1$, then
$$
\absL{\frac{a}{b} - \frac{\hat a}{\hat b}} \le (\abs{a} + \abs{b}) \delta.
$$
\end{enumerate}
\end{lemma}

\begin{prf}
\ref{lem approx bounds square}. We have $\abs{a^2 - \hat a^2} = \abs{a+\hat a} \cdot \abs{a-\hat a} \le 2 \delta$.

\ref{lem approx bounds prod}. We have $\abs{ab-\hat a \hat b} = \abs{a(b - \hat b) + \hat b (a -\hat a)} \le \abs{a} \cdot \abs{b - \hat b} + \abs{\hat b} \cdot \abs{a -\hat a} \le 2 \delta$.


\ref{lem approx bounds ratio}. We have
\begin{align*}
\absL{\frac{a}{b} - \frac{\hat a}{\hat b}} 
& = \absL{\frac{a \hat b - \hat a b}{b \hat b}} 
\le \abs{a \hat b - \hat a b} 
= \abs{a (\hat b - b) + b (a - \hat a)} \\
& \le \abs{a} \cdot \abs{\hat b - b} + \abs{b} \cdot \abs{a - \hat a} 
\le (\abs{a} + \abs{b}) \delta,
\end{align*}
where the first inequality uses $\abs{b \hat b} \ge 1$.
\end{prf}


\subsubsection{Construction}

The next lemma is a rational, algorithmic version of Lemma~\ref{lem exact Jacobi construction}: rather than the exact (generally irrational) cosine-sine pair $(c,s)$ produced by that lemma, it efficiently computes a nearby rational cosine-sine pair $(\hat c,\hat s)$.

\begin{lemma}
\label{lem rational Jacobi approx}
Let $A \in \R^{n \times n}$ be symmetric and let $(p, q)$ be an index pair that satisfies $1 \le p<q \le n$.
Let $c,s$ be defined as in Lemma~\ref{lem exact Jacobi construction}, and let $\delta \in (0,1]$.
There is an algorithm that, on input $a_{pp}, a_{qq}, a_{pq}$ and $\delta$, finds integers $P_1,P_2,Q$ with $Q>0$, $P_1^2+P_2^2=Q^2$, and $\size(Q) \le 2\ceil{\log(1/\delta)}+10$, such that, letting $\hat c := P_1/Q$ and $\hat s := P_2/Q$,
$$
\abs{c-\hat c}\le \delta, \qquad \abs{s - \hat s}\le \delta.
$$
The algorithm performs $O(1+\log(1/\delta))$ arithmetic operations in the real RAM model.
Moreover, if $a_{pp}, a_{qq}, a_{pq}, \delta \in \Q$, the algorithm runs on a Turing machine in time polynomial in $\size(a_{pp}) + \size(a_{qq}) + \size(a_{pq}) + \size(\delta)$.
\end{lemma}

\begin{prf}
If $a_{pq}=0$, then $t=0$, $c=1$, and $s=0$, and the algorithm returns $(P_1,P_2,Q):=(1,0,1)$: then $P_1^2+P_2^2=Q^2$, the bound on $\size(Q)$ holds trivially, and $\hat c = c$, $\hat s = s$.
We may therefore assume $a_{pq}\neq0$.
We compute $\tau \in \R$ as defined in Lemma~\ref{lem exact Jacobi construction}, and we define $d := \ceil{\log(8/\delta)} = 3+\ceil{\log(1/\delta)}$, which is a positive integer since $\delta \le 1$.
Let $\theta, t$ be as in Lemma~\ref{lem exact Jacobi construction} and let
$$
u := \tan(\theta/2) = \frac{\sqrt{1+t^2}-1}{t}.
$$

\smallskip
\noindent
\textbf{Construction of $\hat u$ with $\abs{u - \hat u} \le 2^{-d}$.}
We give an algorithm that takes in input $\tau$, and returns 
$\hat u \in \Q \cap [-1,1]$ with $\size(\hat u) \le 2d+3$ so that 
$$
\abs{u - \hat u} \le 2^{-d}.
$$

First, we consider the case $\tau \ge 0$.
The algorithm is based on a simple binary search. 
Since $t\in(0,1]$, we have $\theta \in (0,\pi/4]$ and hence $u = \tan(\theta/2) \in (0,1)$, thus we start from the interval $[0,1]$.
At each iteration we compare the midpoint $m > 0$ of the interval with $u$ using the conditions on $u \lesseqgtr m$ derived below:
\begin{align*}
u \lesseqgtr m \quad
& \Leftrightarrow \quad
\frac{\sqrt{1+\pare{\sqrt{1+\tau^2}-\tau}^2}-1}{\sqrt{1+\tau^2}-\tau} \lesseqgtr m \\
& \Leftrightarrow \quad
\sqrt{1+\pare{\sqrt{1+\tau^2}-\tau}^2} \lesseqgtr m \sqrt{1+\tau^2} - m \tau + 1 \\
& \Leftrightarrow \quad
1+\pare{\sqrt{1+\tau^2}-\tau}^2 \lesseqgtr \pare{m \sqrt{1+\tau^2} - m \tau + 1}^2 \\
& \Leftrightarrow \quad
2+2\tau^2 -2 \tau \sqrt{1+\tau^2} \lesseqgtr 2 m^2 \tau^2 + m^2 - 2 m \tau  + 1 - 2 m (m \tau - 1) \sqrt{1+\tau^2} \\
& \Leftrightarrow \quad
\underbrace{\pare{2 m^2 \tau - 2 m-2 \tau}}_{=:\alpha} \sqrt{1+\tau^2} \lesseqgtr \underbrace{{2 m^2 \tau^2 + m^2 - 2 m \tau - 2\tau^2 -1}}_{=:\gamma} \\
& \Leftrightarrow \quad
\begin{cases}
-1 \lesseqgtr 1 & \qquad \text{if $\alpha < 0, \ \gamma \ge 0$} \\
1 \lesseqgtr -1 & \qquad \text{if $\alpha \ge 0, \ \gamma < 0$} \\
\alpha^2 \pare{1+\tau^2} \lesseqgtr \gamma^2 & \qquad \text{if $\alpha,\gamma \ge 0$} \\
\gamma^2 \lesseqgtr \alpha^2 \pare{1+\tau^2} & \qquad \text{if $\alpha,\gamma < 0$}.
\end{cases}
\end{align*}
At the end of iteration $d$, we have an interval of length $2^{-d}$, with rational extreme points in $[0,1]$ of the form $p/2^{d}$ with $p \in \bra{0,1,\dots,2^{d}}$.
Since this interval contains $u$, either extreme point is at distance at most $2^{-d}$ from $u$.
We then set $\hat u$ to be any such extreme point.

\smallskip

Next, we consider the case $\tau < 0$.
Again, our algorithm is based on a simple binary search. 
Since $t\in(-1,0)$, we have $\theta \in (-\pi/4,0)$ and hence $u = \tan(\theta/2) \in (-1,0)$, thus we start from the interval $[-1,0]$.
At each iteration we compare the midpoint $m < 0$ of the interval with $u$ using the conditions on $u \lesseqgtr m$ derived below:
\begin{align*}
u \lesseqgtr m \quad
& \Leftrightarrow \quad
\frac{\sqrt{1+\pare{-\tau-\sqrt{1+\tau^2}}^2}-1}{-\tau-\sqrt{1+\tau^2}} \lesseqgtr m \\
& \Leftrightarrow \quad
-m \sqrt{1+\tau^2} - m \tau + 1 \lesseqgtr \sqrt{1+\pare{-\tau-\sqrt{1+\tau^2}}^2} \\
& \Leftrightarrow \quad
\pare{-m \sqrt{1+\tau^2} - m \tau + 1}^2 \lesseqgtr 1+\pare{-\tau-\sqrt{1+\tau^2}}^2 \\
& \Leftrightarrow \quad
2 m^2 \tau^2 + m^2 - 2 m \tau  + 1 - 2 m (1 -m \tau ) \sqrt{1+\tau^2} \lesseqgtr 2+2\tau^2 +2 \tau \sqrt{1+\tau^2} \\
& \Leftrightarrow \quad
\underbrace{{2 m^2 \tau^2 + m^2 - 2 m \tau - 2\tau^2 -1}}_{=:\gamma} \lesseqgtr \underbrace{\pare{2 m + 2 \tau - 2 m^2 \tau}}_{=:\beta} \sqrt{1+\tau^2} \\
& \Leftrightarrow \quad
\begin{cases}
1 \lesseqgtr -1 & \qquad \text{if $\gamma \ge 0, \ \beta < 0$} \\
-1 \lesseqgtr 1 & \qquad \text{if $\gamma < 0, \ \beta \ge 0$} \\
\gamma^2 \lesseqgtr \beta^2 \pare{1+\tau^2} & \qquad \text{if $\gamma,\beta \ge 0$} \\
\beta^2 \pare{1+\tau^2} \lesseqgtr \gamma^2 & \qquad \text{if $\gamma,\beta < 0$}.
\end{cases}
\end{align*}
At the end of iteration $d$, we have an interval of length $2^{-d}$, with rational extreme points in $[-1,0]$ of the form $p/2^{d}$ with $p \in \bra{-2^{d}, -2^{d}+1,\dots,0}$.
Since this interval contains $u$, either extreme point is at distance at most $2^{-d}$ from $u$.
We then set $\hat u$ to be any such extreme point.

In both cases, since $\hat u = p/2^{d}$ for some integer $p$ with $\abs{p}\le 2^{d}$, we have
$$
\size(\hat u) \le 1+\ceil{\log(2^{d}+1)}+\ceil{\log(2^{d}+1)} = 2d+3.
$$

We remark that the above comparisons only involve the rational quantities $\alpha,\beta,\gamma,\tau,m$, and never require evaluating the irrational quantity $\sqrt{1+\tau^2}$ itself; this is what allows the algorithm to run in polynomial time on a Turing machine.

\smallskip
\noindent
\textbf{Construction of $(\hat c,\hat s)$.}
In the remainder of the proof, we denote $\hat\theta := 2 \arctan(\hat u)$ so that $\hat u = \tan(\hat\theta/2)$.
Since $\hat u = p/2^{d}$ with $\abs{p}\le 2^{d}$, define
$$
P_1 := 2^{2d}-p^2, \qquad P_2 := 2p\cdot 2^{d}, \qquad Q := 2^{2d}+p^2,
$$
and let $\hat c := P_1/Q$, $\hat s := P_2/Q$. Note that
$$
\hat c = \cos(\hat\theta) = \frac{1-\hat u^2}{1+\hat u^2}, \qquad \hat s = \sin(\hat\theta) = \frac{2\hat u}{1+\hat u^2},
$$
so $(\hat c,\hat s)$ is the cosine-sine pair associated with $\hat\theta$, approximating
$$
c = \cos(\theta) = \frac{1-u^2}{1+u^2} \in \R, \qquad s = \sin(\theta) = \frac{2u}{1+u^2} \in \R.
$$
Also, $Q>0$, and
$$
P_1^2+P_2^2 = (2^{2d})^2-2\cdot2^{2d}p^2+p^4+4p^2\cdot2^{2d} = (2^{2d})^2+2\cdot2^{2d}p^2+p^4 = (2^{2d}+p^2)^2 = Q^2,
$$
implying $\hat c^2+\hat s^2=1$.
Since $0\le p^2\le2^{2d}$, we have $2^{2d}\le Q\le2\cdot2^{2d}=2^{2d+1}$, so
$$
\size(Q) \le 1+\ceil{\log(2^{2d+1}+1)}+\ceil{\log 2} = 1+(2d+2)+1 = 2d+4 = 2\ceil{\log(1/\delta)}+10.
$$

\smallskip
\noindent
\textbf{Bounding $\abs{c-\hat c}$ and $\abs{s-\hat s}$.}
From Lemma~\ref{lem approx bounds}.\ref{lem approx bounds square} (with $a=u,\hat a=\hat u$), since $u,\hat u\in[-1,1]$ and $\abs{u-\hat u}\le 2^{-d}$, we obtain
$$
\abs{u^2-\hat u^2} \le 2\cdot 2^{-d} = 2^{-d+1}.
$$
From Lemma~\ref{lem approx bounds}.\ref{lem approx bounds ratio} (with $a = 1 - u^2$, $\hat a = 1 - \hat u^2$, $b = 1 + u^2$, $\hat b = 1 + \hat u^2$), we obtain
$$
\abs{c - \hat c} \le 3 \cdot 2^{-d+1} \le 2^{-d+3} \le \delta,
$$
where we used bounds $\abs{1 - u^2} \le 1$, $\abs{1 + u^2} \le 2$ and $\abs{1 + u^2}, \abs{1 + \hat u^2} \ge 1$.
Similarly, from Lemma~\ref{lem approx bounds}.\ref{lem approx bounds ratio} (with $a = 2u$, $\hat a = 2 \hat u$, $b = 1 + u^2$, $\hat b = 1 + \hat u^2$), we obtain
$$
\abs{s - \hat s} \le 4 \cdot 2^{-d+1} = 2^{-d+3} \le \delta,
$$
where we used bounds $\abs{2 u} \le 2$, $\abs{1 + u^2} \le 2$ and $\abs{1 + u^2}, \abs{1 + \hat u^2} \ge 1$.

\smallskip
\noindent
\textbf{Running time.}
Computing $d$ from $\delta$, performing the $d$ binary-search iterations that construct $\hat u$, and constructing $(\hat c,\hat s)$ with a further $O(1)$ arithmetic operations amount to $O(d)=O(1+\log(1/\delta))$ arithmetic operations in total; this establishes the real RAM bound.
Now suppose $a_{pp},a_{qq},a_{pq},\delta\in\Q$. Then $\tau$ has size $O(\size(a_{pp})+\size(a_{qq})+\size(a_{pq}))$, and the midpoint $m$ at iteration $i$ has size $O(i)=O(d)$, so the quantities $\alpha,\beta,\gamma$, obtained from $\tau$ and $m$ by $O(1)$ arithmetic operations, have size $O(\size(a_{pp})+\size(a_{qq})+\size(a_{pq})+\size(\delta))$, where we used $d=O(\size(\delta))$. Since each arithmetic operation on rationals of size $L$ can be performed in time polynomial in $L$, and $d$ is computable from $\delta$ in polynomial time, each of the $O(d)$ iterations and the $O(1)$ arithmetic operations forming $(\hat c,\hat s)$ run in polynomial time; hence the entire algorithm runs on a Turing machine in time polynomial in $\size(a_{pp})+\size(a_{qq})+\size(a_{pq})+\size(\delta)$.
\end{prf}

The next lemma builds on Lemma~\ref{lem rational Jacobi approx}: for a given tolerance $\delta$, we show that the exact condition $b_{pq}=0$ from Lemma~\ref{lem exact Jacobi construction} can be relaxed to $\abs{\hat b_{pq}}\le\delta$, for a suitable rotation $\hat J$.
This is achieved by invoking Lemma~\ref{lem rational Jacobi approx} with a tolerance $\delta'$ obtained from $\delta$ by scaling it down by the magnitude of $a_{pp},a_{qq},a_{pq}$, which is how accurately $\hat c$ and $\hat s$ must approximate $c$ and $s$ in order for $\hat b_{pq}$ to be at most $\delta$.

\begin{lemma}
\label{lem rational Jacobi construction}
Let $A \in \R^{n \times n}$ be symmetric, let $(p, q)$ be an index pair that satisfies $1 \le p<q \le n$, and let $\delta \in (0,1]$.
Let $S := \max\bra{0, \ceil{\log\pare{4\pare{\abs{a_{p q}}+\abs{a_{p p}}+\abs{a_{q q}}}}}}$.
There is an algorithm that, on input $a_{pp},a_{qq},a_{pq}$ and $\delta$, finds integers $P_1,P_2,Q$ with $Q>0$, $P_1^2+P_2^2=Q^2$, and $\size(Q)\le 2\ceil{\log(1/\delta)}+2S+10$, such that, letting $\hat c:=P_1/Q$, $\hat s:=P_2/Q$, $\hat J := J(p, q; \hat c, \hat s)$, and $\hat B:=\hat J^\transp A \hat J$, we have $\abs{\hat b_{pq}} \le \delta$.
The algorithm performs $O(1+\log(1/\delta)+S)$ arithmetic operations in the real RAM model.
Moreover, if $a_{pp},a_{qq},a_{pq},\delta\in\Q$, the algorithm runs on a Turing machine in time polynomial in $\size(a_{pp})+\size(a_{qq})+\size(a_{pq})+\size(\delta)$.
\end{lemma}

\begin{prf}
If $a_{pq}=0$, the algorithm returns $(P_1,P_2,Q):=(1,0,1)$: then $P_1^2+P_2^2=Q^2$, the bound on $\size(Q)$ holds trivially, and $\hat J$ is the identity, so $\hat b_{pq}=a_{pq}=0\le\delta$.
We may therefore assume $a_{pq}\neq0$.
Let $c,s$ be as defined in Lemma~\ref{lem exact Jacobi construction}, let $J:=J(p,q;c,s)$, and $B:=J^\transp A J$.
Define
$$
\delta' := \min\bra{1, \ \frac{\delta}{4\pare{\abs{a_{pq}}+\abs{a_{pp}}+\abs{a_{qq}}}}} \in (0,1]
$$
(note the denominator is positive, since $a_{pq}\neq0$).
We apply Lemma~\ref{lem rational Jacobi approx} with $\delta'$ in place of $\delta$, and obtain integers $P_1,P_2,Q$ with $Q>0$, $P_1^2+P_2^2=Q^2$, and $\size(Q)\le2\ceil{\log(1/\delta')}+10$, such that, letting $\hat c:=P_1/Q$, $\hat s:=P_2/Q$,
$$
\abs{c-\hat c}\le \delta', \qquad \abs{s - \hat s}\le \delta'.
$$
Let $\hat J := J(p,q;\hat c,\hat s)$.

Consider the entries in rows and columns $p,q$ in the matrix products $B=J^\transp A J$ and $\hat B=\hat J^\transp A \hat J$:
\begin{align*}
\begin{bmatrix}
b_{p p} & b_{p q} \\
b_{q p} & b_{q q}
\end{bmatrix}
& =
\begin{bmatrix}
c & s \\
-s & c
\end{bmatrix}^\transp
\begin{bmatrix}
a_{p p} & a_{p q} \\
a_{q p} & a_{q q}
\end{bmatrix}
\begin{bmatrix}
c & s \\
-s & c
\end{bmatrix}, \\
\begin{bmatrix}
\hat b_{p p} & \hat b_{p q} \\
\hat b_{q p} & \hat b_{q q}
\end{bmatrix}
& =
\begin{bmatrix}
\hat c & \hat s \\
-\hat s & \hat c
\end{bmatrix}^\transp
\begin{bmatrix}
a_{p p} & a_{p q} \\
a_{q p} & a_{q q}
\end{bmatrix}
\begin{bmatrix}
\hat c & \hat s \\
-\hat s & \hat c
\end{bmatrix}.
\end{align*}
Hence,
\begin{align*}
b_{p q} & = 0 = a_{p q}\pare{c^2-s^2}+\pare{a_{p p}-a_{q q}} c s, \\
\hat b_{p q} & = a_{p q}\pare{\hat c^2-\hat s^2}+\pare{a_{p p}-a_{q q}} \hat c \hat s.
\end{align*}
Taking the difference of the two equations,
$$
\hat b_{p q} = a_{p q}\pare{(\hat c^2 -c^2)-(\hat s^2-s^2)}+\pare{a_{p p}-a_{q q}} (\hat c \hat s - cs).
$$
Since $c,s,\hat c,\hat s \in [-1,1]$, we can apply Lemma~\ref{lem approx bounds}.\ref{lem approx bounds square} and Lemma~\ref{lem approx bounds}.\ref{lem approx bounds prod} (with $\delta:=\delta'$) to obtain
\begin{align*}
\abs{\hat b_{p q}} 
& \le \abs{a_{p q}}\pare{\abs{\hat c^2 -c^2}+\abs{\hat s^2-s^2}}+\abs{a_{p p}-a_{q q}} \abs{\hat c \hat s - cs} \\
& \le 4 \delta' \abs{a_{p q}}+ 2 \delta' \abs{a_{p p}-a_{q q}} \\
& \le \delta' \pare{4\abs{a_{pq}}+2\abs{a_{pp}}+2\abs{a_{qq}}} \\
& \le \delta' \cdot 4 \pare{\abs{a_{pq}}+\abs{a_{pp}}+\abs{a_{qq}}}.
\end{align*}
By definition of $\delta'$, we have $\delta' \cdot 4\pare{\abs{a_{pq}}+\abs{a_{pp}}+\abs{a_{qq}}} \le \delta$: this is an equality if $\delta' < 1$, and otherwise $\delta'=1$ and $4\pare{\abs{a_{pq}}+\abs{a_{pp}}+\abs{a_{qq}}} \le \delta$. Therefore
$$
\abs{\hat b_{pq}} \le \delta' \cdot 4\pare{\abs{a_{pq}}+\abs{a_{pp}}+\abs{a_{qq}}} \le \delta.
$$
Furthermore, since $1/\delta' = \max\bra{1, \ 4\pare{\abs{a_{pq}}+\abs{a_{pp}}+\abs{a_{qq}}}/\delta}$ and $S \ge \log\pare{4\pare{\abs{a_{pq}}+\abs{a_{pp}}+\abs{a_{qq}}}}$, we have $\log(1/\delta') \le S + \log(1/\delta)$, and hence, $S$ being an integer, $\ceil{\log(1/\delta')} \le S + \ceil{\log(1/\delta)}$; the bound $\size(Q)\le2\ceil{\log(1/\delta)}+2S+10$ follows.

\smallskip
\noindent
\textbf{Running time.}
Besides invoking Lemma~\ref{lem rational Jacobi approx} with $\delta'$, the algorithm only computes $\delta'$ from $a_{pp}$, $a_{qq}$, $a_{pq}$, and $\delta$, which takes $O(1)$ arithmetic operations; since $\log(1/\delta')\le S+\log(1/\delta)$, the total is $O(1+\log(1/\delta')) = O(1+\log(1/\delta)+S)$, establishing the real RAM bound.
Now suppose $a_{pp},a_{qq},a_{pq},\delta\in\Q$. Then $\delta'$ is rational of size polynomial in $\size(a_{pp})+\size(a_{qq})+\size(a_{pq})+\size(\delta)$, and by Lemma~\ref{lem rational Jacobi approx} the entire algorithm runs on a Turing machine in time polynomial in $\size(a_{pp})+\size(a_{qq})+\size(a_{pq})+\size(\delta')$, hence in $\size(a_{pp})+\size(a_{qq})+\size(a_{pq})+\size(\delta)$.
\end{prf}

\subsubsection{Repeated approximate updates}

Let $A \in \R^{n \times n}$ be symmetric, let $(p, q)$ be an index pair that satisfies $1 \le p<q \le n$ and $|a_{p q}| = \max_{i \neq j}|a_{ij}|$, and let $\delta \in (0,1]$.
We say that $\hat B \in \R^{n\times n}$ is obtained by applying a \emph{$\delta$-approximate Jacobi update} to $A$ if $\hat B = \hat J^\transp A \hat J$ for some cosine-sine pair $(\hat c,\hat s)$, where $\hat J:=J(p,q;\hat c,\hat s)$, such that $\abs{\hat b_{pq}} \le \delta$.
A $\delta$-approximate Jacobi update is an approximate version of a Jacobi update, where the exact condition $b_{pq}=0$ is relaxed to $\abs{\hat b_{pq}}\le \delta$.
Such a pair $(\hat c,\hat s)$ can be computed via Lemma~\ref{lem rational Jacobi construction}.

The next two lemmas give approximate versions of Lemma~\ref{lem exact Jacobi sweep}: exact Jacobi updates are replaced with approximate ones, and correspondingly the constant $N$ in the contraction factor $1-1/N$ changes from $\frac{n(n-1)}2$ to $\frac23n(n-1)$, to compensate for the weaker per-step contraction that Lemma~\ref{lem Jacobi norm} yields when $\hat b_{pq}$ is only small rather than zero. Unlike Lemma~\ref{lem exact Jacobi sweep}, whose conclusion is stated in terms of $\off(A^{(k)})$, we state the first lemma in terms of individual off-diagonal entries, since this is the natural form here: $(\delta/2)$-approximate Jacobi updates are precisely what is needed to obtain an entrywise bound of $\delta$. 
The subsequent lemma then recovers a conclusion in terms of $\off(\hat A^{(k)})$, the form needed to match Proposition~\ref{prop exact Jacobi near diag} and to serve as input to Theorem~\ref{th rational Jacobi near diag}.

\begin{lemma}
\label{lem rational Jacobi sweep}
Let $A \in \R^{n \times n}$ be symmetric and let $\delta \in (0,1]$.
For every nonnegative integer $k$, let $\hat A^{(k)}$ denote a matrix obtained from $A$ by $k$ successive $(\delta/2)$-approximate Jacobi updates (with $\hat A^{(0)}:=A$).
Then, there exists 
$$
k \le \max\pare{0, \ \left\lceil \frac23 n(n-1) \ln\pare{\frac{\off(A)^2}{\delta^2}}\right\rceil}
$$
such that all off-diagonal entries of $\hat A^{(k)}$ are at most $\delta$ in absolute value.
\end{lemma}
\begin{prf}
Let $N:=\frac23 n(n-1)$ and $\xi := \ln\pare{\off(A)^2/\delta^2}$.
If $\off(A)^2 < \delta^2$, then every off-diagonal entry of $A$ has absolute value at most $\delta$, and $k=0$ satisfies the claim. We may therefore assume $\off(A)^2 \ge \delta^2$, so that $\xi\ge0$.
Let $K := \ceil{N\xi}$, a nonnegative integer.
If there exists $k \le K-1$ such that all off-diagonal entries of $\hat A^{(k)}$ are at most $\delta$ in absolute value we are done.
We may therefore assume that for each $k \le K-1$ there is at least one off-diagonal entry of $\hat A^{(k)}$ that is larger than $\delta$ in absolute value.
We show that all off-diagonal entries of $\hat A^{(K)}$ are at most $\delta$ in absolute value.
Let $a_{p q}$ be the off-diagonal entry of $A$ with largest absolute value. Then
$$
\off(A)^2 \le n(n-1) a_{p q}^2 = \frac32 N a_{p q}^2.
$$
Since a $(\delta/2)$-approximate Jacobi update satisfies $\abs{\hat b_{pq}} \le \delta/2$, Lemma~\ref{lem Jacobi norm} gives
$$
\off(\hat A^{(1)})^2 
= \off(A)^2- 2 a_{p q}^2 + 2 \hat b_{p q}^2
\le \off(A)^2- 2 a_{p q}^2 + \frac{\delta^2}{2}
\le \off(A)^2- \frac32 a_{p q}^2
\le \pare{1-\frac{1}{N}} \off(A)^2,
$$
where the second inequality uses $2 \hat b_{p q}^2 \le 2 \pare{\delta/2}^2 = \delta^2/2$, the third uses $\abs{a_{pq}} > \delta$, which gives $\delta^2/2 \le a_{p q}^2/2$, and the last uses $\off(A)^2 \le \frac32 N a_{p q}^2$.
Since for every $k \le K-1$ the off-diagonal entry of $\hat A^{(k)}$ with largest absolute value has absolute value larger than $\delta$, by induction,
$$
\off(\hat A^{(K)})^2 \le\pare{1-\frac{1}{N}}^{K} \off(A)^2.
$$
Since $K \ge N\xi$ and $0<1-\frac1N<1$, we have $\pare{1-\frac1N}^K \le \pare{1-\frac1N}^{N\xi}$.
Since $1-\frac1N \le e^{-1/N}$, we also have $\pare{1-\frac1N}^{N\xi} \le e^{-\xi}$.
Thus, 
$$
\off(\hat A^{(K)})^2 \le \pare{1-\frac1N}^K\off(A)^2 \le e^{-\xi} \off(A)^2
= \frac{\delta^2}{\off(A)^2}\off(A)^2
=\delta^2.
$$
Therefore, each off-diagonal entry of $\hat A^{(K)}$ is, in absolute value, at most $\delta$, and $k=K$ satisfies the claim.
\end{prf}

\begin{lemma}
\label{lem rational Jacobi sweep off}
Let $A \in \R^{n \times n}$ be symmetric and let $\delta \in (0,1]$.
For every nonnegative integer $k$, let $\hat A^{(k)}$ denote a matrix obtained from $A$ by $k$ successive $\pare{\delta/(2n)}$-approximate Jacobi updates (with $\hat A^{(0)}:=A$).
Then, there exists 
$$
k \le \max\pare{0, \ \left\lceil \frac23 n(n-1) \ln\pare{\frac{n^2\off(A)^2}{\delta^2}}\right\rceil}
$$
such that $\off(\hat A^{(k)}) \le \delta$.
\end{lemma}
\begin{prf}
Note that $\delta/n \in (0,1]$, since $n \ge 1$.
By Lemma~\ref{lem rational Jacobi sweep} (applied with $\delta/n$ in place of $\delta$, so that its updates are exactly the $\pare{\delta/(2n)}$-approximate ones considered here), there exists
$$
k \le \max\pare{0, \ \left\lceil \frac23 n(n-1) \ln\pare{\frac{n^2\off(A)^2}{\delta^2}}\right\rceil}
$$
such that all off-diagonal entries of $\hat A^{(k)}$ are at most $\delta/n$ in absolute value. Since $\off(\hat A^{(k)})^2$ sums the squares of at most $n(n-1)$ off-diagonal entries, each at most $\delta^2/n^2$,
$$
\off(\hat A^{(k)})^2 \le n(n-1) \frac{\delta^2}{n^2} = \frac{n-1}{n}\,\delta^2 \le \delta^2.
$$
Hence $\off(\hat A^{(k)})\le\delta$, as claimed.
\end{prf}

\subsubsection{Near-diagonalization via rational rotations}

Unlike Proposition~\ref{prop exact Jacobi near diag}, which relies on exact Jacobi rotations and hence cannot be implemented on a Turing machine, the following theorem assembles Lemmas~\ref{lem rational Jacobi construction} and~\ref{lem rational Jacobi sweep off} into an algorithm that drives a symmetric matrix arbitrarily close to diagonal form via rational Jacobi rotations. Unlike Proposition~\ref{prop exact Jacobi near diag}, the output $L$ is rational even when $A$ is only real; when $A$ is rational, the algorithm runs in time polynomial on a Turing machine. This result is the main algorithmic building block for the remainder of the paper.

\begin{theorem}[Near-diagonalization via rational Jacobi rotations]
\label{th rational Jacobi near diag}
Let $A \in \R^{n \times n}$ be symmetric and let $\delta \in (0,1]$.
There is an algorithm that computes an orthogonal matrix $L \in \Q^{n \times n}$ such that $\off(L^\transp A L) \le \delta$.
The algorithm performs $O\pare{n^4\pare{\log(1/\delta)+\log n+\log(\max(1,\Fnorm{A}))}^2}$ arithmetic operations in the real RAM model.
Moreover, if $A\in\Q^{n\times n}$ and $\delta$ is rational, the algorithm runs on a Turing machine in time polynomial in $\size(A) + \size(\delta)$.
\end{theorem}

\begin{prf}
From Lemma~\ref{lem rational Jacobi sweep off}, there exists
$$
k \le \max\pare{0,\ \left\lceil \frac23 n(n-1) \ln\pare{\frac{n^2\off(A)^2}{\delta^2}}\right\rceil},
$$
such that after $k$ successive $\pare{\delta/(2n)}$-approximate Jacobi updates, the obtained matrix $\hat A^{(k)}$ satisfies $\off(\hat A^{(k)})\le\delta$.
Let $\hat J_1,\dots,\hat J_k$ be the corresponding Jacobi rotations, each computed via Lemma~\ref{lem rational Jacobi construction} with tolerance $\delta/(2n)$, and define $L := \hat J_1 \cdots \hat J_k$ (with $L:=I$ if $k=0$), so that $L^\transp A L = \hat A^{(k)}$.
As a product of orthogonal matrices, $L$ is orthogonal; moreover $L\in\Q^{n\times n}$, since each $\hat J_i$ has rational entries (Lemma~\ref{lem rational Jacobi construction}), regardless of whether $A$ is rational.

\smallskip
\noindent
\textbf{Running time.}
We first bound the entries encountered along the sequence. For each $t\le k-1$, since $\hat J_1,\dots,\hat J_t$ are orthogonal we have $\Fnorm{\hat A^{(t)}}=\Fnorm A$; as the squares $a_{pp}^2,a_{qq}^2,a_{pq}^2$ of the entries of $\hat A^{(t)}$ used to construct $\hat J_{t+1}$ are terms of $\Fnorm{\hat A^{(t)}}^2$, each of $\abs{a_{pp}},\abs{a_{qq}},\abs{a_{pq}}$ is at most $\Fnorm A$. Hence the quantity $S$ of Lemma~\ref{lem rational Jacobi construction} at step $t+1$, namely $S_{t+1}:=\max\bra{0,\ceil{\log(4(\abs{a_{pp}}+\abs{a_{qq}}+\abs{a_{pq}}))}}$, satisfies $S_{t+1}\le S^*:=\max\bra{0,\ceil{\log(12\Fnorm A)}}=O(1+\log(\max(1,\Fnorm A)))$, using $\abs{a_{pp}}+\abs{a_{qq}}+\abs{a_{pq}}\le3\Fnorm A$. Write $\sigma := \log(2n/\delta)+\log(\max(1,\Fnorm A))$. By Lemma~\ref{lem rational Jacobi sweep off} and $\off(A)\le\Fnorm A$,
$$
k = O\pare{n^2 \sigma}.
$$

Note that $S^*=O(\sigma)$. At each step $t$, finding the pivot takes $O(n^2)$ comparisons, constructing $\hat J_{t+1}$ takes $O(\log(2n/\delta)+S_{t+1})=O(\sigma)$ arithmetic operations (Lemma~\ref{lem rational Jacobi construction} with tolerance $\delta/(2n)$), and forming $\hat A^{(t+1)}=\hat J_{t+1}^\transp\hat A^{(t)}\hat J_{t+1}$ takes $O(n^2)$ more, since only rows and columns $p,q$ change, for $O(n^2+\sigma)$ per step. Since $\sigma\ge1$, the total is
$$
O\pare{k(n^2+\sigma)} = O\pare{n^2\sigma(n^2+\sigma)} = O\pare{n^4\sigma^2};
$$
substituting the definition of $\sigma$ gives the claimed bound
$$
O\pare{n^4\pare{\log(1/\delta)+\log n+\log(\max(1,\Fnorm A))}^2}
$$
and establishes the real RAM bound.

Now suppose $A\in\Q^{n\times n}$ and $\delta$ is rational. Let $\deno_0$ be the product of the denominators of the entries of $A$ and $\tilde A_0:=\deno_0A\in\Z^{n\times n}$, so $\size(\deno_0),\ \max_{i,j}\size((\tilde A_0)_{ij})=O(\size(A))$. By Lemma~\ref{lem rational Jacobi construction} (with tolerance $\delta/(2n)$), we may write $\hat J_t=\tilde J_t/Q_t$ with $Q_t\in\Z_{>0}$, $\size(Q_t)=O(\sigma)$, and $\tilde J_t\in\Z^{n\times n}$ satisfying $\abs{(\tilde J_t)_{ij}}\le Q_t$. Setting $\tilde A_t:=\tilde J_t^\transp\tilde A_{t-1}\tilde J_t\in\Z^{n\times n}$ and $\deno_t:=Q_t^2\deno_{t-1}$, we have $\hat A^{(t)}=\tilde A_t/\deno_t$ for every $t$. Since $\Fnorm A=2^{O(\size(A))}$, $n^2=O(\size(A))$, and $\log(2n/\delta)=O(\size(A)+\size(\delta))$, both $\sigma$ and $k=O(n^2\sigma)$ are polynomial in $\size(A)+\size(\delta)$. Since the size of a product of positive integers is at most the sum of their sizes up to an additive $O(1)$, and $\tilde A_t$ is obtained from $\tilde A_{t-1}$ by $O(1)$ integer additions and multiplications involving entries of $\tilde J_t$ (each of size $O(\sigma)$), induction on $t\le k$ gives
$$
\size(\deno_t),\ \max_{i,j}\size((\tilde A_t)_{ij}) = O\pare{\size(A)+k\sigma},
$$
which is polynomial in $\size(A)+\size(\delta)$. Hence every entry of $\hat A^{(t)}=\tilde A_t/\deno_t$ has size polynomial in $\size(A)+\size(\delta)$, so each of the $k$ steps (finding the pivot, constructing $\hat J_{t+1}$, and forming $\tilde A_{t+1}$ and $\deno_{t+1}$) runs on a Turing machine in time polynomial in $\size(A)+\size(\delta)$. As $k$ is also polynomial in $\size(A)+\size(\delta)$, the entire algorithm runs on a Turing machine in time polynomial in $\size(A)+\size(\delta)$.
\end{prf}

\section{Polynomial-time diagonalization}
\label{sec diagonalization}

The goal of this section is to compute, in polynomial time, an orthogonal matrix that diagonalizes a symmetric rational matrix up to a prescribed precision, while preserving its inertia (Theorem~\ref{th diagonalize}); we then extend this to the simultaneous diagonalization of a symmetric matrix and a positive definite matrix (Theorem~\ref{th two matrices}), and to simultaneous diagonalization with respect to a rational polytope (Theorem~\ref{th simultaneous diagonalization}). The main obstacle is that a symmetric matrix may have eigenvalues arbitrarily close to $0$, which could cause a small perturbation to change the inertia; we first show (Lemmas~\ref{lem char 1}--\ref{lem noise}) that this cannot happen for rational matrices, since their nonzero eigenvalues are bounded away from $0$ by an amount that depends only on the size of the matrix.

\subsection{Eigenvalue gaps}

Consider an $n \times n$ matrix $A$.
The \emph{characteristic polynomial} of $A$, denoted by ${\displaystyle p_{A}(t),}$ is the polynomial defined by
$$
p_A(t)=\det(t I-A).
$$ 
It is well known that the roots of $p_A(t)$ are exactly the eigenvalues of $A$, and that, if $A$ is symmetric, all eigenvalues of $A$ are real.
Throughout the paper, for a symmetric matrix $A\in\R^{n\times n}$, we denote by $\lambda_k(A)$ its $k$-th largest eigenvalue, so that $\lambda_1(A)\ge\lambda_2(A)\ge\cdots\ge\lambda_n(A)$ (counted with multiplicity); when the matrix is clear from the context, we simply write $\lambda_1,\dots,\lambda_n$.
In particular, if $A$ is symmetric,
$$
p_A(t) = (t-\lambda_1)(t-\lambda_2)\cdots(t-\lambda_n).
$$
For a polynomial $p(t)$, we denote by $\pnorm{p}$ the maximum absolute value of the coefficients of $p$.
We first bound the coefficients of the characteristic polynomial of a symmetric matrix in terms of $n$ and its Frobenius norm.

\begin{lemma}
\label{lem char 1}
Let $A \in \R^{n \times n}$ be symmetric and let $p_A(t)$ be its characteristic polynomial.
Then, 
$$
\pnorm{p_A} \le 2^n \max(1,\Fnorm{A})^n.
$$
Furthermore, if $A \in \Z^{n \times n}$, then the coefficients of $p_A$ are integer.
\end{lemma}

\begin{prf}
It follows directly from the definition of the characteristic polynomial that, if $A \in \Z^{n \times n}$, then the coefficients of $p_A$ are integer.
Let $\lambda_1, \dots, \lambda_n \in \R$ be the eigenvalues of $A$. As noted above,
$$
p_A(t) = t^n+a_{n-1} t^{n-1}+ \cdots +a_1 t + a_0 = (t-\lambda_1) (t-\lambda_2) \cdots (t-\lambda_n).
$$
For $i=0,\dots,n-1$, the number $a_i$ is then
$$
a_i = \sum_{\mathcal I \subseteq \{1,\dots,n\} : \abs{\mathcal I} = i} \prod_{j \notin \mathcal I} (- \lambda_j).
$$
Since $|\lambda_j|\le\Fnorm{A}\le\max(1,\Fnorm A)$ for every $j$ (as $\Fnorm{A}^2=\sum_i\lambda_i^2$, so each $\lambda_j^2$ is one term in this sum), each of the $\binom{n}{i}$ terms in the sum satisfies
$$
\abs{\prod_{j\notin \mathcal I}(-\lambda_j)} = \prod_{j\notin \mathcal I}\abs{\lambda_j} \le \max(1,\Fnorm A)^{n-i} \le \max(1,\Fnorm{A})^n,
$$
where the last inequality holds since $\max(1,\Fnorm A)\ge1$ and $n-i\le n$.
Hence,
\[
\abs{a_i} \le \binom{n}{i}\max(1,\Fnorm A)^n \le 2^n \max(1,\Fnorm{A})^n. 
\]
The leading coefficient of $p_A$ is $1$, which also satisfies this bound.
\end{prf}

Using the bound in Lemma~\ref{lem char 1}, we now show that every nonzero eigenvalue of an integer symmetric matrix is bounded away from $0$, by an amount that depends only on $n$ and the Frobenius norm of the matrix.
We will be using the following lemma, due to Cauchy~\cite{Cau29}, and taken from Mignotte~\cite{Mig82}.

\begin{lemma}[\cite{Cau29,Mig82}]
\label{lem Cauchy}
Let $q(t) \in \R[t]$ be the nonzero polynomial
$$
q(t)=b_d t^d+\cdots+b_1 t+b_0.
$$
Let $\lambda$ be a complex root of $q$. Then
$$
\abs{\lambda} \ge \frac{\abs{b_0}}{\abs{b_0}+\max\pare{\abs{b_1}, \dots, \abs{b_d}}}.
$$
\end{lemma}

\begin{lemma}
\label{lem eigen zero gap}
Let $A \in \Z^{n \times n}$ be symmetric.
Then, for every nonzero eigenvalue $\lambda$ of $A$, we have
$$
\abs{\lambda} \ge \frac{1}{1 + 2^n \max(1,\Fnorm{A})^n}.
$$
\end{lemma}

\begin{prf}
Let the characteristic polynomial of $A$ be
$$
p_A(t) = t^n+a_{n-1} t^{n-1}+ \cdots +a_1 t + a_0, \qquad a_n:=1.
$$
If $A=0$, the result is trivial (there is no nonzero eigenvalue), so we assume $A\neq0$.

Let $k\in\{0,1,\dots,n-1\}$ be the multiplicity of $0$ as a root of $p_A(t)$, equivalently, as an eigenvalue of $A$ (we have $k\le n-1$ since $A$ is symmetric and $A\neq0$, so not all eigenvalues of $A$ are zero). Write $p_A(t)=t^k q(t)$, where
$$
q(t) = t^{n-k}+b_{n-k-1}t^{n-k-1}+\cdots+b_1t+b_0, \qquad b_i:=a_{i+k} \text{ for } i=0,\dots,n-k.
$$
Since the eigenvalues of $A$ are the roots of $p_A(t)=t^kq(t)$, with exactly $k$ of them equal to $0$, the roots of $q(t)$ are exactly the nonzero eigenvalues of $A$; in particular, $q(0)=b_0\neq0$. Since $A\in\Z^{n\times n}$, the coefficients of $p_A$ are integer by Lemma~\ref{lem char 1}; in particular $a_k$ (hence $b_0$) is a nonzero integer, so $\abs{b_0}\ge1$.
From Lemma~\ref{lem Cauchy} (applied to $q(t)$), we have that each root $\lambda$ of $q(t)$ satisfies
$$
\abs{\lambda} 
\ge \frac{\abs{b_0}}{\abs{b_0} + \max\bra{\abs{b_1},\dots,\abs{b_{n-k}}}}.
$$
Bounding each $\abs{b_i}$ through Lemma~\ref{lem char 1}, and using $\abs{b_0}\ge1$, we obtain
\begin{align*}
\abs{\lambda} 
& \ge \frac{\abs{b_0}}{\abs{b_0} + \max\bra{\abs{b_1},\dots,\abs{b_{n-k}}}} 
\ge \frac{\abs{b_0}}{\abs{b_0} + 2^n \max(1,\Fnorm{A})^n} \\
& \ge \frac{1}{1 + 2^n \max(1,\Fnorm{A})^n},
\end{align*}
where the second inequality uses $b_i=a_{i+k}$ for $i=1,\dots,n-k$ together with Lemma~\ref{lem char 1}.
Since every nonzero eigenvalue $\lambda$ of $A$ is a root of $q(t)$, this proves the claim.
\end{prf}

\subsection{Rounding near-diagonal matrices}


The next result shows that, given an integer symmetric matrix $A$ and an orthogonal $L$ such that $L^\transp AL$ is already nearly diagonal, we can safely round to zero its off-diagonal entries and its small diagonal entries, without changing the inertia of $A$ or moving its eigenvalues by more than a controlled amount; the proof relies on the eigenvalue gap established in Lemma~\ref{lem eigen zero gap}.
We will also be using the following classical eigenvalue perturbation result.

\begin{lemma}[Corollary 8.1.6 in \cite{GolVan4th}]
\label{lem eigenvalue perturbation}
If $B, C \in \R^{n\times n}$ are symmetric, then, for $i=1,\dots,n$,
$$
\abs{\lambda_i(B)-\lambda_i(C)} \le \norm{B-C}.
$$
\end{lemma}

\begin{lemma}
\label{lem noise}
Let $A \in \Z^{n \times n}$ be symmetric, let $L \in \R^{n \times n}$ be orthogonal, and let $\delta \ge 0$ be such that $\off(L^\transp A L)\le\delta$ and
$$
\delta \le \frac{1}{4\pare{1+2^n\max(1,\Fnorm{A})^n}}.
$$
Let $D \in \R^{n \times n}$ be the diagonal matrix obtained from $L^\transp A L$ by replacing all off-diagonal entries with zero and all diagonal entries with absolute value at most $\delta$ with zero.
Then, $A$ and $D$ have the same inertia.
Furthermore, for $i=1,\dots,n$, we have $\abs{\lambda_i(A) - \lambda_i(D)} \le \delta.$
\end{lemma}
\begin{prf}
Since $L$ is orthogonal, $L^\transp=L^{-1}$; in particular $L$ is nonsingular, so the eigenvalues of $A$ coincide with those of $L^{-1}AL=L^\transp AL$.
From Lemma~\ref{lem eigen zero gap}, every nonzero eigenvalue $\lambda$ of $A$ satisfies 
$$
\abs{\lambda} \ge \frac{1}{1+2^n\max(1,\Fnorm{A})^n} = 4\cdot\pare{\frac{1}{4\pare{1+2^n\max(1,\Fnorm{A})^n}}} \ge 4\delta.
$$
Let $D_0 \in \R^{n \times n}$ be the diagonal matrix obtained from $L^\transp A L$ by replacing all off-diagonal entries with zero.
Note that $D_0$ and $L^\transp AL$ are both symmetric, since $D_0$ is diagonal, and $A$ is symmetric.
Since $L^\transp AL - D_0$ has zero diagonal and off-diagonal entries matching those of $L^\transp AL$, we have $\Fnorm{L^\transp AL - D_0} = \off(L^\transp AL) \le \delta$.
Hence, from Lemma~\ref{lem eigenvalue perturbation} (with $B:=L^\transp AL$ and $C:=D_0$), we have, for every $i=1,\dots,n$,
$$
\abs{\lambda_i(A) - \lambda_i(D_0)} 
= \abs{\lambda_i(L^\transp A L) - \lambda_i(D_0)} 
\le \norm{L^\transp A L-D_0} \le \Fnorm{L^\transp A L-D_0} \le \delta.
$$
Thus: $\lambda_i(A)>0$ implies $\lambda_i(A) \ge 4\delta$ and $\lambda_i(D_0) \ge 3\delta$, $\lambda_i(A)<0$ implies $\lambda_i(A) \le -4\delta$ and $\lambda_i(D_0) \le -3\delta$, and $\lambda_i(A)=0$ implies $\abs{\lambda_i(D_0)} \le \delta$. 
Let $D \in \R^{n \times n}$ be the diagonal matrix obtained from $D_0$ by replacing all diagonal entries with absolute value at most $\delta$ with zero.
Since $D_0$ is diagonal, its eigenvalues are its diagonal entries, and by the above each of them has absolute value either at least $3\delta$ or at most $\delta$.
Replacing those of the second kind by $0$ therefore does not change the relative order of the diagonal entries, so $\lambda_i(D)$ is obtained from $\lambda_i(D_0)$ by the same replacement: $\abs{\lambda_i(D_0)} \ge 3\delta$ implies $\lambda_i(D) = \lambda_i(D_0)$, and $\abs{\lambda_i(D_0)} \le \delta$ implies $\lambda_i(D) = 0$.
Therefore, for every $i$, $\lambda_i(A)$ and $\lambda_i(D)$ are both positive, both zero, or both negative; since these cases are exhaustive, $A$ and $D$ have the same inertia.
Finally, if $\lambda_i(D) = \lambda_i(D_0)$, then $\abs{\lambda_i(A)-\lambda_i(D)} = \abs{\lambda_i(A)-\lambda_i(D_0)} \le \delta$; if $\lambda_i(D)=0$, then $\abs{\lambda_i(D_0)}\le\delta$ implies $\lambda_i(A)=0$, so $\abs{\lambda_i(A)-\lambda_i(D)} = \abs{0-0}=0\le\delta$. In either case, $\abs{\lambda_i(A)-\lambda_i(D)}\le\delta$, as claimed.
\end{prf}

\subsection{Diagonalization}

We now combine Theorem~\ref{th rational Jacobi near diag} with Lemmas~\ref{lem eigen zero gap} and~\ref{lem noise} to obtain our main diagonalization result: given a symmetric rational matrix $A$ and a tolerance $\delta$, we compute in polynomial time a rational orthogonal $L$ and a rational diagonal $D$ with the same inertia as $A$, such that $L^\transp AL$ differs from $D$ by a symmetric matrix of Frobenius norm at most $\delta$.

\begin{theorem}[Diagonalization]
\label{th diagonalize}
Let $A \in \Q^{n \times n}$ be symmetric and let $\delta \in (0,1]$ be rational.
There is an algorithm that computes an orthogonal matrix $L \in \Q^{n \times n}$, a diagonal matrix $D \in \Q^{n \times n}$, and a symmetric matrix $E \in \Q^{n \times n}$ such that 
(i) $L^\transp A L = D + E$ with $\Fnorm{E} \le \delta$, 
and
(ii) $A$ and $D$ have the same inertia.
The algorithm runs on a Turing machine in time polynomial in $\size (A) + \size(\delta)$.

Furthermore, there exists $\zeta\in\Z_{>0}$, depending only on $A$ (not on $\delta$), with $\log(\zeta)$ polynomial in $\size(A)$, such that every nonzero entry of $D$ is at least $1/\zeta$ in absolute value; moreover, $\zeta$ can be computed on a Turing machine in time polynomial in $\size(A)$.
\end{theorem}

\begin{prf}
Let $\deno \in \Z_{>0}$ be the product of the denominators of the entries in $A$ and define $\tilde A := \deno A \in \Z^{n \times n}$.
The size of each entry of $\tilde A$ is $O(\size(A))$, thus $\size(\tilde A) = O(n^2 \size(A))$.
By Lemma~\ref{lem eigen zero gap}, every nonzero eigenvalue of $\tilde A$ has absolute value at least
$$
\beta(\tilde A) := \frac{1}{1+2^n\max\pare{1,\Fnorm{\tilde A}}^n} \ge \frac1\omega, \qquad \omega := 1+2^n\max\pare{1,\Fnorm{\tilde A}^2}^n,
$$
where the inequality holds since $\max(1,\Fnorm{\tilde A})\le\max(1,\Fnorm{\tilde A}^2)$, and $\omega\in\Z_{>0}$ because $\Fnorm{\tilde A}^2=\sum_{i,j}\tilde a_{ij}^2$ is a nonnegative integer.
Let 
$$
\delta' := \min\bra{\frac{\delta}{2n}, \ \frac{1}{4\omega}} \in (0,1],
$$
so that $2n \delta' \le \delta$ and $\delta' \le 1/(4\omega) \le \beta(\tilde A)/4$; note that $\delta'$ is rational, of size polynomial in $\size(A)+\size(\delta)$.
We employ Theorem~\ref{th rational Jacobi near diag} (with $A:=\tilde A$ and $\delta:=\delta'$) and obtain an orthogonal matrix $L \in \Q^{n \times n}$ such that $\off(L^\transp \tilde A L) \le \delta'$.
Since $\delta'\le\beta(\tilde A)/4=\frac{1}{4\pare{1+2^n\max(1,\Fnorm{\tilde A})^n}}$, we may apply Lemma~\ref{lem noise} (with $A:=\tilde A$, $\delta:=\delta'$).
Let $\tilde D \in \Q^{n \times n}$ be the diagonal matrix obtained from $L^\transp \tilde A L$ by replacing all off-diagonal entries with zero, and all diagonal entries with absolute value at most $\delta'$ with zero.
By that lemma, $\tilde A$ and $\tilde D$ have the same inertia and, for $i=1,\dots,n$,
\begin{equation}
\label{eq shown before}
\abs{\lambda_i(\tilde A) - \lambda_i(\tilde D)} \le \delta'.
\end{equation}
Let $\tilde E := L^\transp \tilde A L - \tilde D$.
The off-diagonal entries of $\tilde E$ coincide with those of $L^\transp \tilde AL$, so the sum of their squares equals $\off(L^\transp \tilde AL)^2\le(\delta')^2$; by the definition of $\tilde D$, the diagonal entries of $\tilde E$ have absolute value at most $\delta'$.
Thus, 
$$
\Fnorm{\tilde E}^2 \le (\delta')^2 + n (\delta')^2 = (1+n)(\delta')^2 \le 2n (\delta')^2,
$$
so $\Fnorm{\tilde E} \le \sqrt{2n}\,\delta' \le 2n \delta' \le \delta$, where we used $\sqrt{2n}\le 2n$.
Now let $D := \tilde D/\deno \in \Q^{n \times n}$ and $E := \tilde E/\deno \in \Q^{n \times n}$.
Clearly $D$ is diagonal, $E$ is symmetric, and 
$$
L^\transp A L = \frac 1\deno L^\transp \tilde A L = \frac 1\deno (\tilde D + \tilde E) = D + E.
$$
Since $\deno$ is positive, $A$ and $\tilde A$ have the same inertia, and so do $\tilde D$ and $D$; since $\tilde A$ and $\tilde D$ have the same inertia, this proves (ii).
Moreover, $\deno \ge 1$ implies 
$$
\Fnorm{E} = \frac 1\deno \Fnorm{\tilde E} \le \Fnorm{\tilde E} \le \delta.
$$
This proves (i).
\smallskip
\noindent
\textbf{An a priori lower bound on the nonzero diagonal entries of $D$.}
Let $\zeta := 2\deno \omega \in \Z_{>0}$.
Since $\Fnorm{\tilde A}^2$ is a nonnegative integer of size polynomial in $\size(A)$ (as $\size(\tilde A)=O(n^2\size(A))$) and $\size(\deno)$ is polynomial in $\size(A)$, $\log(\zeta)$ is polynomial in $\size(A)$ and $\zeta$ is computable in time polynomial in $\size(A)$; note also that $\zeta$ depends only on $A$, not on $\delta$.
For every $i$ with $\lambda_i(A)\ne0$, we have $\lambda_i(\tilde A)=\deno\lambda_i(A)\ne0$, so $|\lambda_i(\tilde A)|\ge\beta(\tilde A)$ by Lemma~\ref{lem eigen zero gap}. Hence, using $|\lambda_i(\tilde A)-\lambda_i(\tilde D)|\le\delta'$ (shown in~\eqref{eq shown before}),
$$
\abs{\lambda_i(\tilde D)} \ge \abs{\lambda_i(\tilde A)} - \abs{\lambda_i(\tilde A)-\lambda_i(\tilde D)} \ge \beta(\tilde A) - \delta'.
$$
Since $\delta'\le\beta(\tilde A)/4$ (shown above), this gives
$$
\abs{\lambda_i(\tilde D)} \ge \beta(\tilde A) - \frac{\beta(\tilde A)}{4} = \frac34\beta(\tilde A).
$$
Dividing by $\deno\ge1$ and using $\beta(\tilde A)\ge 1/\omega$,
$$
\abs{\lambda_i(D)} = \frac{\abs{\lambda_i(\tilde D)}}{\deno} \ge \frac{3\beta(\tilde A)}{4\deno} \ge \frac{3}{4\deno \omega} = \frac32\cdot\frac1{\zeta} \ge \frac1{\zeta}.
$$
Thus, every nonzero eigenvalue $\lambda_i(D)$ satisfies $\abs{\lambda_i(D)}\ge1/\zeta$. Since $D$ is diagonal, its diagonal entries coincide, as a multiset, with its eigenvalues $\lambda_1(D),\dots,\lambda_n(D)$; hence every nonzero diagonal entry of $D$ satisfies the same bound, establishing the claim.

\smallskip
\noindent
\textbf{Running time.}
Every quantity other than $L$ is obtained by $O(n^2)$ arithmetic operations on integers or rationals of size polynomial in $\size(A)+\size(\delta)$: $\deno$ and $\tilde A=\deno A$ have size polynomial in $\size(A)$, as do $\Fnorm{\tilde A}^2$ and $\omega$, while $\size(\delta')$ is polynomial in $\size(A)+\size(\delta)$. By Theorem~\ref{th rational Jacobi near diag} (applied to $\tilde A$ with $\delta:=\delta'$), computing $L$ takes time polynomial in $\size(\tilde A)+\size(\delta')$, hence in $\size(A)+\size(\delta)$, since $\size(\tilde A)=O(n^2\size(A))$. Forming $\tilde D,\tilde E$ from $L^\transp\tilde AL$ and rescaling by $\deno$ to obtain $D,E$ take a further $O(n^2)$ arithmetic operations on rationals of size polynomial in $\size(A)+\size(\delta)$. Hence the entire algorithm runs on a Turing machine in time polynomial in $\size(A)+\size(\delta)$.
\end{prf}

Note that Theorem~\ref{th diagonalize} recovers the conclusion of Theorem~\ref{th rational Jacobi near diag}, for the same $\delta$: since $D$ is diagonal, the off-diagonal entries of $L^\transp AL$ coincide with those of $E$, so $\off(L^\transp AL)=\off(E)\le\Fnorm E\le\delta$.

Theorem~\ref{th diagonalize} also yields an approximate spectral decomposition of $A$.

\begin{corollary}[Approximate spectral decomposition]
\label{cor approx spectral}
Let $A \in \Q^{n \times n}$ be symmetric, let $\delta \in (0,1]$ be rational, and let $L,D,E$ be as in Theorem~\ref{th diagonalize}.
Denote by $v^i$ the $i$th column of $L$ and by $d_i$ the $i$th diagonal entry of $D$.
Then, for every $i=1,\dots,n$,
(i) $\abs{\lambda_i(A) - \lambda_i(D)} \le \delta$,
and
(ii) $\norm{Av^i - d_i v^i} \le \delta$.
\end{corollary}

\begin{prf}
Since $L$ is orthogonal, $A$ and $L^\transp A L = D+E$ are similar, hence $\lambda_i(A) = \lambda_i(D+E)$ for every $i$.
By Lemma~\ref{lem eigenvalue perturbation}, applied to $D+E$ and $D$, we obtain
$$
\abs{\lambda_i(A) - \lambda_i(D)} = \abs{\lambda_i(D+E) - \lambda_i(D)} \le \norm{E} \le \Fnorm{E} \le \delta,
$$
which proves (i).
For (ii), let $e^i$ denote the $i$th vector of the standard basis of $\R^n$, so that $v^i = Le^i$ and $L^{-1}v^i = e^i$, since $L$ is orthogonal and thus $L^{-1} = L^\transp$.
From $L^\transp A L = D+E$ we get $A = L(D+E)L^{-1}$, so
$$
Av^i = L(D+E)L^{-1}v^i = L(D+E)e^i = d_iv^i + LEe^i,
$$
and therefore $\norm{Av^i-d_iv^i} = \norm{LEe^i} = \norm{Ee^i} \le \Fnorm{E} \le \delta$.
\end{prf}

By (ii), each column $v^i$ of $L$ behaves as an approximate eigenvector of $A$, with corresponding approximate eigenvalue $d_i$, while by (i) the diagonal entries of $D$ approximate the eigenvalues of $A$.
Since the columns $v^1,\dots,v^n$ of the orthogonal matrix $L$ form an orthonormal basis of $\R^n$, this yields a full set of $n$ simultaneous approximate eigenvectors, each with residual controlled by the same tolerance $\delta$.
In particular, denoting by $L_\delta, D_\delta, E_\delta$ the matrices obtained for the tolerance $\delta$, we have $\Fnorm{E_\delta}\to0$ as $\delta\to0$ by Theorem~\ref{th diagonalize}(i); since the set of orthogonal matrices is compact, every limit point of $L_\delta$ as $\delta\to0$ is an exact eigenvector matrix of $A$, and the corresponding limit point of $D_\delta$ is the diagonal matrix of the eigenvalues of $A$, consistently with the classical spectral theorem for symmetric matrices.

\subsection{Simultaneous diagonalization of two matrices}

We now consider the simultaneous diagonalization of a symmetric matrix $A$ and a positive definite matrix $M$: classically, one seeks a nonsingular matrix $L$ such that $L^\transp A L$ is diagonal and $L^\transp M L = I$.
As shown by Example~\ref{ex no rational L}, such an $L$ is in general irrational, hence it cannot be computed on a Turing machine.
The next theorem shows that, if the requirement that $L^\transp A L$ be diagonal is relaxed to being diagonal up to a perturbation of Frobenius norm at most $\delta$, then $L$ can be taken rational and can be computed in polynomial time, while the condition $L^\transp M L = I$ is still satisfied exactly and the inertia of $A$ is still preserved.
We assume that $M$ is given in the factored form $M = C^\transp C$ for a rational matrix $C$; the remark following the theorem shows that this assumption cannot be dropped.

\begin{theorem}[Simultaneous diagonalization of two matrices]
\label{th two matrices}
Let $A \in \Q^{n \times n}$ be symmetric, let $C \in \Q^{n \times n}$ be nonsingular, let $M := C^\transp C$, and let $\delta \in (0,1]$ be rational.
There is an algorithm that computes a nonsingular matrix $L \in \Q^{n \times n}$, a diagonal matrix $D \in \Q^{n \times n}$, and a symmetric matrix $E \in \Q^{n \times n}$ such that
(i) $L^\transp A L = D + E$ with $\Fnorm{E} \le \delta$,
(ii) $A$ and $D$ have the same inertia,
and
(iii) $L^\transp M L = I$.
The algorithm runs on a Turing machine in time polynomial in $\size(A)+\size(C)+\size(\delta)$.

Furthermore, there exists $\zeta \in \Z_{>0}$, depending only on $A$ and $C$ (not on $\delta$), with $\log(\zeta)$ polynomial in $\size(A)+\size(C)$, such that every nonzero entry of $D$ is at least $1/\zeta$ in absolute value; moreover, $\zeta$ can be computed on a Turing machine in time polynomial in $\size(A)+\size(C)$.
\end{theorem}

\begin{prf}
Let $\tilde A := C^{-\transp} A C^{-1}$, which is symmetric and belongs to $\Q^{n\times n}$, and which has the same inertia as $A$ by Sylvester's law of inertia.
Since $\tilde A$ is obtained from $A$ and $C$ by a constant number of matrix inversions and multiplications, $\size(\tilde A)$ is polynomial in $\size(A)+\size(C)$.
We employ Theorem~\ref{th diagonalize} (with $A:=\tilde A$ and the same $\delta$) and obtain an orthogonal matrix $\tilde L \in \Q^{n \times n}$, a diagonal matrix $D \in \Q^{n \times n}$, and a symmetric matrix $E \in \Q^{n \times n}$ such that $\tilde L^\transp \tilde A \tilde L = D+E$, $\Fnorm{E} \le \delta$, and $\tilde A$ and $D$ have the same inertia.
Define $L := C^{-1} \tilde L \in \Q^{n \times n}$, which is nonsingular.
Then $L^\transp A L = \tilde L^\transp C^{-\transp} A C^{-1} \tilde L = \tilde L^\transp \tilde A \tilde L = D+E$, which proves (i), and (ii) holds since $A$, $\tilde A$, and $D$ all have the same inertia.
Moreover
$$
L^\transp M L = \tilde L^\transp C^{-\transp} \pare{C^\transp C} C^{-1} \tilde L = \tilde L^\transp \tilde L = I,
$$
since $\tilde L$ is orthogonal, which proves (iii).
The ``Furthermore'' clause follows from the corresponding clause of Theorem~\ref{th diagonalize}, applied with $A:=\tilde A$, together with the fact that $\size(\tilde A)$ is polynomial in $\size(A)+\size(C)$.

\smallskip
\noindent
\textbf{Running time.}
Forming $\tilde A$ and $L$ takes a constant number of matrix operations on rationals of size polynomial in $\size(A)+\size(C)$, and, by Theorem~\ref{th diagonalize}, computing $\tilde L, D, E$ takes time polynomial in $\size(\tilde A)+\size(\delta)$.
Hence the entire algorithm runs on a Turing machine in time polynomial in $\size(A)+\size(C)+\size(\delta)$.
\end{prf}

Both departures of Theorem~\ref{th two matrices} from the classical statement \eqref{eq classical simultaneous} are necessary.
First, by Example~\ref{ex no rational L}, the first requirement in \eqref{eq classical simultaneous} cannot be met by a rational $L$, even when $M = I$; this is why condition (i) is a relaxation.
Second, the hypothesis that $M$ be given in the factored form $M = C^\transp C$, with $C$ rational, cannot be dropped, as the next example shows.

\begin{example}
\label{ex M not factored}
Let $n := 2$ and $M := 3 I$, and suppose that $L^\transp M L = I$ for some rational nonsingular $L$.
Then $L^\transp L = \frac13 I$, so the first column $(x,y)$ of $L$ satisfies $x^2+y^2 = \frac13$, and thus $(3x)^2+(3y)^2 = 3$, exhibiting $3$ as a sum of two rational squares.
Writing $3 = (p/r)^2+(q/r)^2$ with $p,q,r$ integers such that $\gcd(p,q,r)=1$, we obtain $3r^2 = p^2+q^2$; since every square is congruent to $0$ or $1$ modulo $3$, it follows that $p \equiv q \equiv 0 \pmod 3$, hence $9$ divides $3r^2$, so that $3$ divides $r$, contradicting $\gcd(p,q,r)=1$.
Therefore no rational nonsingular $L$ satisfies $L^\transp M L = I$; equivalently, there is no nonsingular rational $C$ with $M = C^\transp C$, since $L := C^{-1}$ would then satisfy $L^\transp M L = I$.
\end{example}

Example~\ref{ex M not factored} exhibits a further obstruction, beyond the irrationality of Example~\ref{ex no rational L}, to performing the classical simultaneous reduction on a Turing machine.

\subsection{Simultaneous diagonalization and rounding}

We now extend Theorem~\ref{th two matrices} to the setting where, in place of the matrix $M$, we are given a full-dimensional rational polytope, and we wish the matrix $L$ to transform this polytope into one sandwiched between two concentric balls (an ellipsoid rounding condition). This will later allow us to reduce a general mixed integer quadratic program to one with a separable objective function and posed over a polytope with a spherical form.

We will be using the following algorithmic result, due to Lenstra~\cite{Len83} (see also Corollary~15.6a in Schrijver~\cite{SchBookIP}).
Given a nonsingular matrix $C \in \R^{n \times n}$ and a vector $a \in \R^n$, we define the \emph{ellipsoid}
$$
\E(a,C) := \{ x \in \R^n : \norm{C (x-a)} \le 1\}.
$$

\begin{lemma}[\cite{Len83}, Corollary 15.6a in \cite{SchBookIP}]
\label{lem Lenstra sandwich}
There is an algorithm that, given a rational system $Wx\le w$ such that $\P :=\{x\in\R^n:Wx\le w\}$ is full-dimensional and bounded, computes a nonsingular matrix $C\in\Q^{n\times n}$ and a vector $a\in\Q^n$ such that
$$
\E(a,C) \subseteq \P \subseteq \E\pare{a,C/n^{3/2}},
$$
in time polynomially bounded by $\size(W)+\size(w)$.
\end{lemma}

\begin{theorem}[Simultaneous diagonalization and rounding]
\label{th simultaneous diagonalization}
Let $A \in \Q^{n \times n}$ be symmetric, let $\delta \in (0,1]$ be rational, and let $Wx\le w$ be a rational system such that $\P := \{x \in \R^n: Wx \le w\}$ is full-dimensional and bounded.
There is an algorithm that computes a nonsingular matrix $L \in \Q^{n \times n}$, a diagonal matrix $D \in \Q^{n \times n}$, a symmetric matrix $E \in \Q^{n \times n}$, and $a \in \Q^n$ such that 
(i) $L^\transp A L = D + E$ with $\Fnorm{E} \le \delta$, 
(ii) $A$ and $D$ have the same inertia, 
and 
(iii) $\E (a, L^{-1}) \subseteq \P \subseteq \E (a, L^{-1} / n^{3/2})$.
The algorithm runs on a Turing machine in time polynomial in $\size(A)+\size(W)+\size(w) + \size(\delta)$.

Furthermore, there exists $\zeta \in \Z_{>0}$, depending only on $A,W,w$ (not on $\delta$), with $\log(\zeta)$ polynomial in $\size(A)+\size(W)+\size(w)$, such that every nonzero entry of $D$ is at least $1/\zeta$ in absolute value; moreover, $\zeta$ can be computed on a Turing machine in time polynomial in $\size(A)+\size(W)+\size(w)$.
\end{theorem}

\begin{prf}
By Lemma~\ref{lem Lenstra sandwich}, there is an algorithm, running on a Turing machine in time polynomial in $\size(W)+\size(w)$, that computes a nonsingular matrix $C \in \Q^{n \times n}$ and $a \in \Q^n$ such that
\begin{align}
\label{eq elli cont pos}
\E (a, C) \subseteq \P \subseteq \E (a,C / n^{3/2});
\end{align}
in particular, $\size(C)$ and $\size(a)$ are polynomial in $\size(W)+\size(w)$.
We employ Theorem~\ref{th two matrices} (with this $A$ and $C$, the same $\delta$, and $M:=C^\transp C$) and obtain a nonsingular matrix $L \in \Q^{n \times n}$, a diagonal matrix $D \in \Q^{n \times n}$, and a symmetric matrix $E \in \Q^{n \times n}$ such that $L^\transp M L = I$, $L^\transp A L = D+E$ with $\Fnorm{E} \le \delta$, and $A$ and $D$ have the same inertia.
This establishes (i) and (ii).

It remains to prove (iii).
From $L^\transp M L = I$ we obtain $M = L^{-\transp} L^{-1}$, and therefore, for every $x \in \R^n$,
\begin{align*}
\norm{L^{-1} (x - a)}^2 = (x-a)^\transp M (x-a) = \norm{C (x - a)}^2,
\end{align*}
from which we derive $\E (a, L^{-1}) = \E (a, C)$ and $\E (a, L^{-1} / n^{3/2}) = \E (a, C / n^{3/2})$.
Condition (iii) then follows from \eqref{eq elli cont pos}.

\smallskip
\noindent
\textbf{An a priori lower bound on the nonzero diagonal entries of $D$.}
By the ``Furthermore'' clause of Theorem~\ref{th two matrices}, there exists $\zeta\in\Z_{>0}$, depending only on $A$ and $C$ (not on $\delta$), with $\log(\zeta)$ polynomial in $\size(A)+\size(C)$, such that every nonzero entry of $D$ is at least $1/\zeta$ in absolute value, and $\zeta$ can be computed on a Turing machine in time polynomial in $\size(A)+\size(C)$.
Since $\size(C)$ is polynomial in $\size(W)+\size(w)$, both $\log(\zeta)$ and the time to compute $\zeta$ are polynomial in $\size(A)+\size(W)+\size(w)$; note also that $\zeta$ depends only on $A,W,w$ (via $C$), not on $\delta$.

\smallskip
\noindent
\textbf{Running time.}
Computing $C$ and $a$ takes time polynomial in $\size(W)+\size(w)$ (Lemma~\ref{lem Lenstra sandwich}), and forming $M=C^\transp C$ takes a constant number of matrix operations.
By Theorem~\ref{th two matrices}, computing $L,D,E$ takes time polynomial in $\size(A)+\size(C)+\size(\delta)$, hence in $\size(A)+\size(W)+\size(w)+\size(\delta)$.
Hence the entire algorithm runs on a Turing machine in time polynomial in $\size(A)+\size(W)+\size(w)+\size(\delta)$.
\end{prf}

\section{Approximating mixed integer quadratic programming}
\label{sec MIQP}

In this section we prove Theorem~\ref{th MIQP algorithm}.
We first introduce \emph{spherical form} MIQP, a special form of Problem~\ref{prob MIQP} in which the objective function is separable and the polyhedron described by the linear constraints is sandwiched between two concentric balls.
Theorem~\ref{th simultaneous diagonalization} allows us to reduce Problem~\ref{prob MIQP} to this form, and we give an approximation algorithm for spherical form MIQP, based on mesh partitioning and linear underestimators.
Both the reduction and the algorithm are subject to the same proviso: the objective function must vary enough on the feasible region.
Two theorems of the alternative then remove this proviso: applied to an arbitrary instance, they either produce an approximate solution, certifying in the process that the objective function does vary enough, 
or split the instance into sub-instances with one fewer integer variable.
A recursion on these two outcomes yields Theorem~\ref{th MIQP algorithm}.

\subsection{Spherical form MIQP}

To define a spherical form MIQP, we first introduce mixed integer lattices and balls.
Given linearly independent vectors $b^1, \dots , b^n \in \R^n$ and an integer $p \in \{0,\dots,n\}$, we define the \emph{mixed integer lattice} 
$$
\Lambda_p(b^1, \dots , b^n) := \bra{\sum_{i=1}^n \mu_i b^i : \mu_i \in \Z \ \forall i=1,\dots,p, \ \mu_i \in \R \ \forall i=p+1,\dots,n}.
$$
A \emph{ball} with center $a$ and radius $r$ is a set of the form
\begin{align*}
\B(a,r) := \{y \in \R^n : \norm{y-a} \le r\},
\end{align*}
where $a \in \R^n$ and $r > 0$.
Note that $\B(a,r) = \E(a,I / r)$.

A \emph{spherical form MIQP} is an optimization problem of the form
\begin{align}
\label{prob SMIQP}
\tag{S-MIQP}
\begin{split}
\min & \quad y^\transp D y + h^\transp y  \\
\st & \quad Wy \le w  \\
& \quad y \in \Lambda_p(b^1,\dots,b^n) + \{c\},
\end{split}
\end{align}
where $D$ is diagonal and the polyhedron $\{y \in \R^n : Wy \le w\}$ has a \emph{spherical form}:
\begin{align*}
\B(0,1) \subseteq \{y \in \R^n : Wy \le w\} \subseteq \B(0,n^{3/2}).
\end{align*}
In Problem~\ref{prob SMIQP}, $y \in \R^n$ is the vector of variables and the input data consists of a diagonal matrix $D \in \Q^{n \times n}$, vectors $h,c \in \Q^n$, a matrix $W \in \Q^{m \times n}$, a vector $w \in \Q^m$, an integer $p \in \{0,1,\dots,n\}$, and linearly independent vectors $b^1,\dots,b^n \in \Q^n$.
Given an instance of Problem~\ref{prob SMIQP} (i.e., specific data $D,h,W,w,b^1,\dots,b^n,c$), we define its \emph{size} as $\size(D) + \size(h) + \size(W) + \size(w) + \sum_{i=1}^n \size(b^i) + \size(c)$.
Note that the feasible region of Problem~\ref{prob SMIQP} is compact: it is contained in $\B(0,n^{3/2})$, and it is closed because $\Lambda_p(b^1,\dots,b^n)+\{c\}$ is closed.
In particular, if it is nonempty, the objective function is bounded on it and attains its minimum and its maximum.
The two key properties of a spherical form MIQP are thus that its objective function is separable, and that its linear inequalities describe a polytope with a spherical form. The price paid for these properties is that the constraint $x \in \Z^p \times \R^{n-p}$ of Problem~\ref{prob MIQP} is replaced by the less elegant constraint $y \in \Lambda_p(b^1,\dots,b^n) + \{c\}$.
We remark that, due to the different type of lattice constraint, Problem~\ref{prob SMIQP} is technically not of the form~\eqref{prob MIQP}, however it is simple to check that the change of variables $y = B x + c$, where $B$ is the matrix with columns $b^1,\dots,b^n$, maps Problem~\ref{prob SMIQP} into a problem of the form~\eqref{prob MIQP}.

Our definition differs substantially from the spherical form MIQP considered in~\cite{dP23bMPA}, whose shape is dictated by the fixed rank assumption made there: the variables are split into a first block, of fixed dimension, which carries the quadratic part of the objective function and all the integrality constraints, and a second block, on which the objective function is linear and the variables are continuous; the sphericity condition is then imposed not on the polytope itself, but on its orthogonal projection onto the space of the first block, with a constant ratio between the radii of the two balls.
Problem~\ref{prob SMIQP} requires no such splitting: the objective function is separable in all $n$ variables, and the entire polyhedron has a spherical form.
The price is that the ratio between the two radii is $n^{3/2}$ rather than a constant.
In this respect our spherical form is closer to the one used by Vavasis in the continuous setting~\cite{Vav92i}; the difference is that here the feasible region also carries a lattice constraint.

\subsection{Reduction to spherical form}

We need a lemma that quantifies how sensitive approximate solutions are to perturbations of the objective function.

\begin{lemma}[Approximation of perturbed optimization problems]
\label{lem noise approx 1}
Let $\S$ be a nonempty set and let $f,g: \S \to \R$ be such that $f_{\inf} := \inf_{y \in \S} f(y)$ and $f_{\sup} := \sup_{y \in \S} f(y)$ are finite, and define
\begin{align}
\label{prob opt}
& \min \bra{ f(y) : y \in \S }, \\
\label{prob opt hat}
& \min \bra{ f(y) + g(y) : y \in \S }.
\end{align}
Let $y^\diamond$ be an $\epsilon/2$-approximate solution to Problem~\eqref{prob opt}, for $\epsilon \in [0,1]$, and assume $\abs{g(y)} \le \epsilon (f_{\sup} - f_{\inf})/8$ for every $y \in \S$.
Then $y^\diamond$ is an $\epsilon$-approximate solution to Problem~\eqref{prob opt hat}.
\end{lemma}

\begin{prf}
Let $\hat f(y) := f(y) + g(y)$ and $\Delta := \epsilon (f_{\sup} - f_{\inf})/8$.
Let $\hat f_{\inf} := \inf_{y \in \S} \hat f(y)$ and $\hat f_{\sup} := \sup_{y \in \S} \hat f(y)$, which are both finite since $\abs{g(y)}\le\Delta$ for every $y \in \S$.
First, note that
\begin{align*}
\hat f_{\sup} - \hat f_{\inf} 
& \ge f_{\sup} - f_{\inf} - 2 \Delta \\
& = (1- \epsilon/4) (f_{\sup} - f_{\inf}),
\end{align*}
where we used $\hat f_{\sup}\ge f_{\sup}-\Delta$ and $\hat f_{\inf}\le f_{\inf}+\Delta$.
We obtain
\begin{align*}
\hat f(y^\diamond) - \hat f_{\inf} 
& \le f(y^\diamond) - f_{\inf} + 2 \Delta \\
& \le \epsilon(f_{\sup} - f_{\inf})/2 + \epsilon (f_{\sup} - f_{\inf})/4 \\
& = (1/2 + 1/4) \epsilon (f_{\sup} - f_{\inf}) \\
& \le \frac{1/2 + 1/4}{1-\epsilon/4} \epsilon (\hat f_{\sup} - \hat f_{\inf}) \\
& \le \epsilon (\hat f_{\sup} - \hat f_{\inf}),
\end{align*}
where the first inequality uses $\hat f(y^\diamond)\le f(y^\diamond)+\Delta$ and $\hat f_{\inf}\ge f_{\inf}-\Delta$, the second uses that $y^\diamond$ is an $\epsilon/2$-approximate solution to Problem~\eqref{prob opt} together with $2\Delta = \epsilon(f_{\sup}-f_{\inf})/4$, the fourth uses the inequality above, and the last holds for $\epsilon \in [0,1]$.
Thus, $y^\diamond$ is an $\epsilon$-approximate solution to Problem~\eqref{prob opt hat}.
\end{prf}

We are now ready to show that Problem~\ref{prob MIQP} can be reduced to a spherical form MIQP.
Unlike in the corresponding reduction in \cite{dP23bMPA}, the obtained spherical form MIQP is not equivalent to the original problem; instead, approximate solutions to the spherical form MIQP yield approximate solutions to the original problem, provided that the objective function varies enough over the feasible region.

\begin{proposition}[Reduction to spherical form MIQP]
\label{prop reduction to SMIQP}
Consider an instance of Problem~\ref{prob MIQP} such that $\P := \{x \in \R^n: Wx \le w\}$ is full-dimensional and bounded and $H$ is nonzero, let $\epsilon\in(0,1]$ be rational, and let $\kappa$ be a positive rational absolute constant.
There is an algorithm that computes a nonsingular matrix $L \in \Q^{n \times n}$, a diagonal matrix $D \in \Q^{n \times n}$, and a vector $a \in \Q^n$ such that:
(i)
$H$ and $D$ have the same inertia.
(ii) The following problem is a spherical form MIQP, where $b^1,\dots,b^n$ denote the columns of $L^{-1}$:
\begin{align}
\label{prob from MIQP to SMIQP}
\begin{split}
\min & \quad y^\transp D y + (h^\transp L + 2 a^\transp HL) y \\
\st & \quad WLy \le w - Wa \\
& \quad y \in \Lambda_p(b^1,\dots,b^n) - \{L^{-1} a\}.
\end{split}
\end{align}
(iii)
Let $f$ denote the objective function of Problem~\eqref{prob from MIQP to SMIQP}, let $f_{\inf}$ and $f_{\sup}$ be its infimum and supremum on the feasible region of that problem, and let $\gamma_{\min}$ be the smallest absolute value among the nonzero entries of $D$, which exists since $H$ is nonzero and by (i).
If $y^\diamond$ is an $\epsilon/2$-approximate solution to Problem~\eqref{prob from MIQP to SMIQP} and $f_{\sup} - f_{\inf} \ge \kappa \gamma_{\min}$, then $x^\diamond := L y^\diamond + a$ is an $\epsilon$-approximate solution to Problem~\ref{prob MIQP}.
The algorithm runs on a Turing machine in time polynomial in the size of the instance of Problem~\ref{prob MIQP} and in $\size(\epsilon)$.
\end{proposition}

\begin{prf}
By the ``Furthermore'' clause of Theorem~\ref{th simultaneous diagonalization} (applied with $A:=H$, the given $W,w$, and any rational $\delta\in(0,1]$), there exists $\zeta\in\Z_{>0}$, depending only on $H,W,w$ (not on $\delta$), with $\log(\zeta)$ polynomial in the size of the instance, such that, for every choice of $\delta$, every nonzero entry of the resulting diagonal matrix $D$ is at least $1/\zeta$ in absolute value;
moreover, $\zeta$ can be computed on a Turing machine in time polynomial in $\size(H)+\size(W)+\size(w)$.

We set
$$
\delta := \min\bra{1, \ \frac{\epsilon \kappa}{8 n^3 \zeta}} \in (0,1],
$$
and employ Theorem~\ref{th simultaneous diagonalization} with $A := H$, the given $W,w$, and this $\delta$, obtaining a nonsingular matrix $L \in \Q^{n \times n}$, a diagonal matrix $D \in \Q^{n \times n}$, a symmetric matrix $E \in \Q^{n \times n}$, and $a \in \Q^n$ such that 
$L^\transp H L = D + E$, 
$H$ and $D$ have the same inertia, 
$\Fnorm{E} \le \delta$,
and
$
\E (a, L^{-1}) \subseteq \P \subseteq \E (a, L^{-1} / n^{3/2}).
$
Since $H$ is nonzero and symmetric, it has a nonzero eigenvalue; since $H$ and $D$ have the same inertia, $D$ also has a nonzero diagonal entry, so $\gamma_{\min}$ is well-defined, and, by the choice of $\zeta$ above, $\gamma_{\min}\ge1/\zeta$.

With the change of variables $y = L^{-1} (x-a)$, we obtain
\begin{align*}
\B(0,1) \subseteq \{y \in \R^n : WLy \le w - Wa \} \subseteq \B(0,n^{3/2}).
\end{align*}
Since the objective function of Problem~\eqref{prob from MIQP to SMIQP} is separable (as $D$ is diagonal) and the polyhedron described by its linear constraints has this spherical form, Problem~\eqref{prob from MIQP to SMIQP} is a spherical form MIQP, establishing (ii).
At this point, we have proven (i) and (ii) and have obtained all matrices and vectors that appear as data in Problem~\eqref{prob from MIQP to SMIQP}. 
The remainder of the proof is devoted to proving (iii).

Let $y^\diamond$ be an $\epsilon/2$-approximate solution to Problem~\eqref{prob from MIQP to SMIQP}.
Our goal is to show that $x^\diamond := L y^\diamond + a$ is an $\epsilon$-approximate solution to Problem~\ref{prob MIQP}.
Since the definition of~$\epsilon$-approximate solution is preserved under changes of variables and translations of the objective function, it suffices to show that $y^\diamond$ is an $\epsilon$-approximate solution to the following problem, obtained from Problem~\ref{prob MIQP} by first performing the change of variables $y = L^{-1} (x-a)$, and then by dropping the constant $a^\transp H a + h^\transp a$ in the objective function:
\begin{align}
\label{prob MIQPy}
\begin{split}
\min & \quad y^\transp (D+E) y + (h^\transp L + 2 a^\transp HL) y \\
\st & \quad WLy \le w - Wa \\
& \quad y \in \Lambda_p(b^1,\dots,b^n) - \{L^{-1} a\}.
\end{split}
\end{align}
To obtain the last constraint, note that 
$Ly + a \in \Z^p \times \R^{n-p}$
if and only if
$y \in L^{-1} (\Z^p \times \R^{n-p}) - \{L^{-1} a\}$
and
\begin{align*}
L^{-1} (\Z^p \times \R^{n-p})
& =
L^{-1} \bra{\sum_{i=1}^n \mu_i e^i : \mu_i \in \Z \ \forall i=1,\dots,p, \ \mu_i \in \R \ \forall i=p+1,\dots,n} \\
& =
\bra{\sum_{i=1}^n \mu_i L^{-1} e^i : \mu_i \in \Z \ \forall i=1,\dots,p, \ \mu_i \in \R \ \forall i=p+1,\dots,n} \\
& =
\bra{\sum_{i=1}^n \mu_i b^i : \mu_i \in \Z \ \forall i=1,\dots,p, \ \mu_i \in \R \ \forall i=p+1,\dots,n} \\
& =
\Lambda_p(b^1,\dots,b^n).
\end{align*}

The only difference between Problem~\eqref{prob from MIQP to SMIQP} and Problem~\eqref{prob MIQPy} is the presence of the quadratic function $y^\transp E y$ in the objective function of Problem~\eqref{prob MIQPy}.
We now bound the absolute value of this quadratic function, for all $y \in \B(0,n^{3/2})$, as follows:
\begin{align*}
\abs{y^\transp E y} 
& \le \norm{y}^2 \norm{E} \le n^3 \Fnorm{E} \le n^3 \delta \\
& \le \epsilon \kappa/(8 \zeta) \le \epsilon \kappa \gamma_{\min}/8 \le \epsilon (f_{\sup} - f_{\inf})/8,
\end{align*}
where the third inequality holds by definition of $\delta$ (it is an equality if $\delta = \epsilon\kappa/(8n^3\zeta)$, and otherwise $\delta=1\le\epsilon\kappa/(8n^3\zeta)$), and the second-to-last inequality uses $\gamma_{\min} \ge 1/\zeta$, established above.

The feasible region of Problem~\eqref{prob from MIQP to SMIQP} is nonempty, since it contains $y^\diamond$, and compact, since Problem~\eqref{prob from MIQP to SMIQP} is a spherical form MIQP; hence $f_{\inf}$ and $f_{\sup}$ are finite.
We now apply Lemma~\ref{lem noise approx 1} with \eqref{prob opt} := \eqref{prob from MIQP to SMIQP} and \eqref{prob opt hat} := \eqref{prob MIQPy}.
We obtain that $y^\diamond$ is an $\epsilon$-approximate solution to Problem~\eqref{prob MIQPy}, which, as observed above, implies that $x^\diamond$ is an $\epsilon$-approximate solution to Problem~\ref{prob MIQP}.
This proves (iii).

\smallskip
\noindent
\textbf{Running time.}
As noted above, $\zeta$ is computable in time polynomial in the size of the instance, and $\log(\zeta)$ is polynomial in the size of the instance; hence $\delta=\min\{1,\epsilon\kappa/(8n^3\zeta)\}$ is computable in polynomial time and is a rational of size polynomial in the size of the instance and in $\size(\epsilon)$ (as $\kappa$ is a fixed constant). By Theorem~\ref{th simultaneous diagonalization}, computing $L,D,E,a$ takes time polynomial in the size of the instance and in $\size(\delta)$, hence in the size of the instance and in $\size(\epsilon)$; computing the columns $b^1,\dots,b^n$ of $L^{-1}$ takes a constant number of further matrix operations. Hence the entire algorithm runs on a Turing machine in time polynomial in the size of the instance and in $\size(\epsilon)$.
\end{prf}

\subsection{Approximation algorithm for Problem~\ref{prob SMIQP}}

In this section we show that we can find an $\epsilon$-approximate solution to Problem~\ref{prob SMIQP}, provided that the objective function varies enough on the feasible region.
The algorithm and its analysis are not new: they follow closely a well-established line of work on mesh partitioning and linear underestimators, and we present them here in the precise form needed later.
The technique was applied in the continuous case by Vavasis~\cite{Vav92c,Vav92i}, and in the mixed integer case in~\cite{dP16IPCO,dP18MPB,dP23bMPA}; our treatment borrows from both, with only minor modifications.
As in the mixed integer papers, the feasible region carries a lattice constraint, which is absent in the continuous setting.
As in~\cite{Vav92i}, we linearize the objective function only in the concave directions, leaving the remaining, already convex, quadratic terms unchanged; as a result, the subproblems are mixed integer convex quadratic programs, rather than the mixed integer linear programs obtained in~\cite{dP16IPCO,dP18MPB,dP23bMPA}, and we solve them using the polynomial-time algorithm established in~\cite{dP25SIOPT}.

Formally, a \emph{mixed integer convex quadratic program (MICQP)} is a problem of the form~\eqref{prob MIQP} whose objective function is convex, or, equivalently, whose matrix $H$ is positive semidefinite.
We will make use of the following consequence of \cite{dP25SIOPT}.

\begin{lemma}[\cite{dP25SIOPT}]
\label{lem MICQP}
Assume that, in Problem~\ref{prob MIQP}, the matrix $H$ is positive semidefinite and $\{x \in \R^n: Wx \le w\}$ is bounded.
There is an algorithm that decides whether Problem~\ref{prob MIQP} is infeasible, or has an optimal solution, and in the latter case finds an optimal solution.
The algorithm runs on a Turing machine in time polynomial in the size of the instance, provided that the number $p$ of integer variables is fixed.
\end{lemma}

We will also use the following standard lemma, whose short proof we include for completeness.


\begin{lemma}
\label{lem 5}
Let $q(\lambda)=a \lambda^2+b \lambda+c$ be a univariate quadratic function with $a \le 0$, and let $\ell, u \in \R$ with $\ell \le u$. Let $\tilde q(\lambda)$ be the affine univariate function that attains the same values as $q$ at $\ell, u$. Then, for every $\lambda \in [\ell,u]$,
$$
\tilde q(\lambda) \le q(\lambda) \le \tilde q(\lambda) + \abs{a}(u-\ell)^2/4.
$$
\end{lemma}

\begin{prf}
The function $q-\tilde q$ is a quadratic with leading coefficient $a$ that vanishes at $\ell$ and at $u$, hence $q(\lambda)-\tilde q(\lambda) = a(\lambda-\ell)(\lambda-u) = \abs{a}(\lambda-\ell)(u-\lambda)$.
For $\lambda\in[\ell,u]$ both factors are nonnegative, so this quantity is at least $0$, and their product is at most $((u-\ell)/2)^2$, so it is at most $\abs{a}(u-\ell)^2/4$.
\end{prf}

To simplify the notation, in the remainder of the paper we denote the objective function of Problem~\ref{prob SMIQP} by 
\begin{align*}
\objS(y) := y^\transp D y + h^\transp y.
\end{align*}
We also denote by $\objSinf$, $\objSsup$ the infimum and supremum of $\objS$ on the feasible region of Problem~\ref{prob SMIQP}.

\begin{proposition}[Approximation algorithm for Problem~\ref{prob SMIQP}]
\label{prop SMIQP algorithm}
Let $\kappa$ be a positive rational absolute constant and let $\epsilon\in(0,1]$ be rational.
Assume that Problem~\ref{prob SMIQP} is feasible, that $D$ has at least one negative entry, and that $\objSsup - \objSinf \ge \kappa \gamma_{\mathrm{neg}}$, where $\gamma_{\mathrm{neg}}$ is the largest absolute value among the negative entries of $D$.
There is an algorithm that finds an $\epsilon$-approximate solution to Problem~\ref{prob SMIQP}.
The algorithm runs on a Turing machine in time polynomial in the size of the instance, in $\size(\epsilon)$, and in $1/\epsilon$, provided that $p$ and the number $k$ of negative entries of $D$ are fixed.
\end{proposition}

\begin{prf}
In this proof, we denote by $d_i$ the $i$th diagonal entry of $D$.
After possibly renaming the variables, we assume without loss of generality that $d_1 \le d_2 \le \cdots \le d_n$.
We denote by $k \ge 1$ the number of negative entries of $D$, thus we have $d_1 \le \cdots \le d_k < 0 \le d_{k+1} \le \cdots \le d_n$.
We then have $\gamma_{\mathrm{neg}} = \max_{i=1,\dots,k}\abs{d_i}.$

\smallskip
\noindent
\textbf{The algorithm.}
We start by describing the approximation algorithm.
Let $r := n^2$ (so that $\B(0,n^{3/2})\subseteq\B(0,r)$, since $n^{3/2}\le n^2$ for $n\ge1$).
Next, we define $\varphi^k$ boxes in the first $k$ coordinates of $\R^n$, where $\varphi := \ceilL{r\sqrt{k/(\kappa\epsilon)}}$.
Namely, for each $j_1,\dots,j_k \in \{1,\dots, \varphi\}$, we define the box
\begin{align}
\label{eq boxes}
\C_{j_1,\dots,j_k} := 
\bra{y \in \R^n : -r + \frac{2r}{\varphi} (j_i - 1) \le y_i \le -r + \frac{2r}{\varphi} j_i, \ \forall i =1,\dots,k}.
\end{align}
Note that the union of these $\varphi^k$ boxes is the polyhedron
\begin{align*}
\{y \in \R^n : -r \le y_i \le r, \ \forall i =1,\dots,k\},
\end{align*}
which contains $\B(0,r)$.

We now fix one box $\C = \prod_{i=1}^k [\low_i,\upp_i] \times \R^{n-k}$ among those defined in \eqref{eq boxes}.
For the box $\C$, we construct the affine functions $\linearized_i : \R \to \R$ that attain the same values as $d_i y_i^2$
at $\low_i,\upp_i$, for $i=1,\dots,k$:
\begin{align*}
\linearized_i(y_i) := d_i(\low_i+\upp_i) y_i - d_i \low_i \upp_i 
\qquad \forall i = 1,\dots, k.
\end{align*}
Then we define the affine function $\linearized: \R^{k} \to \R$ given by
\begin{align}
\label{eq underestimator}
\linearized(y_1,\dots,y_k) := \sum_{i=1}^k \linearized_i(y_i),
\end{align}
and we formulate the optimization problem obtained from Problem~\ref{prob SMIQP} by substituting, in the objective function, $\sum_{i=1}^k d_iy_i^2$ with $\linearized(y_1,\dots,y_k)$ and adding the $2k$ linear inequalities in \eqref{eq boxes} that enforce $y \in \C$:
\begin{align}
\label{prob MICQP on box}
\begin{split}
\min & \quad 
\sum_{i=1}^k \pare{d_i(\low_i+\upp_i) y_i - d_i \low_i \upp_i}
+ \sum_{i=k+1}^n d_i y_i^2
+ h^\transp y \\
\st & \quad Wy \le w  \\
& \quad
\ell_i \le y_i \le u_i \qquad \forall i =1,\dots,k \\
& \quad y \in \Lambda_p(b^1,\dots,b^n) + \{c\}.
\end{split}
\end{align}
We show that the change of variables $y = B x + c$, where $B$ is the matrix with columns $b^1,\dots,b^n$, maps Problem~\eqref{prob MICQP on box} to a MICQP defined on a bounded polyhedron, so that we can apply Lemma~\ref{lem MICQP}.
First, this change of variables maps the constraint $y \in \Lambda_p(b^1,\dots,b^n) + \{c\}$ to the constraint $x \in \Z^p \times \R^{n-p}$.
Second, since $B$ is nonsingular, the map $y=Bx+c$ is an invertible affine transformation, and so sends the polyhedron $\{y \in \R^n : Wy \le w,\ \ell_i \le y_i \le u_i \ \forall i=1,\dots,k\}$ to a polyhedron in $x$, namely $\{x \in \R^n : WBx \le w-Wc,\ \ell_i \le (Bx+c)_i \le u_i \ \forall i=1,\dots,k\}$.
Note that the latter polyhedron is bounded, since the polyhedron $\{y \in \R^n: Wy \le w\}$ is already bounded, by the spherical form of Problem~\ref{prob SMIQP}.
Finally, the Hessian of the objective function of Problem~\eqref{prob MICQP on box}, in the variable $y$, is the diagonal matrix with $0$ in each of the first $k$ diagonal entries (since these terms have been linearized) and $d_i\ge0$ in each of the remaining diagonal entries, for $i=k+1,\dots,n$; this matrix is positive semidefinite.
Since $B$ is nonsingular, the Hessian of the objective function of Problem~\eqref{prob MICQP on box} in the variable $x$ is obtained from this matrix by a congruence transformation, and hence, by Sylvester's law of inertia, is positive semidefinite as well.
This shows that Problem~\eqref{prob MICQP on box}, expressed in the variable $x$, is indeed a MICQP.

We now apply Lemma~\ref{lem MICQP} to each Problem~\eqref{prob MICQP on box}.
Since Problem~\ref{prob SMIQP} is feasible and the boxes constructed in \eqref{eq boxes} cover $\B(0,r)$, and hence the entire feasible region of Problem~\ref{prob SMIQP}, at least one of the MICQPs \eqref{prob MICQP on box} is feasible, and therefore has an optimal solution by Lemma~\ref{lem MICQP}.
Among the MICQPs \eqref{prob MICQP on box} that are feasible, the approximation algorithm selects one whose optimal value is smallest, and returns an optimal solution $y^\diamond$ of that MICQP.
This concludes the description of the algorithm.

\smallskip
\noindent
\textbf{Approximation guarantee.}
Next, we show that $y^\diamond$ is an $\epsilon$-approximate solution to 
Problem~\ref{prob SMIQP}.

We first derive an upper bound on $\objS(y^\diamond) - \objSinf$ that depends linearly on $\gamma_{\mathrm{neg}}$.
Let $\C \subseteq \R^n$ be a box constructed in \eqref{eq boxes}, say 
$\C = \prod_{i=1}^k [\low_i,\upp_i] \times \R^{n-k}$.
For each $i = 1,\dots, k,$ we apply Lemma~\ref{lem 5} to each univariate quadratic function $d_i y_i^2$ (using $d_i \le 0$ for $i=1,\dots,k$) and points $\low_i, \upp_i$.
Since $\upp_i - \low_i = 2 r / \varphi$ and $\abs{d_i} \le \gamma_{\mathrm{neg}}$ for $i=1,\dots,k$, we obtain that, for every $y \in \C$,
\begin{align*}
\linearized_i(y_i) & \le d_i y_i^2 \le \linearized_i(y_i) + \gamma_{\mathrm{neg}} r^2 /\varphi^2.
\end{align*}
We sum up all these inequalities for $i=1,\dots,k$ and obtain that for every $y \in \C$,
\begin{align}
\label{eq claim sub}
\linearized(y_1,\dots,y_k) \le 
\sum_{i=1}^k d_i y_i^2
\le \linearized(y_1,\dots,y_k) + \gamma_{\mathrm{neg}} k r^2 / \varphi^2.
\end{align}
Let $\C^\diamond$ be the box among \eqref{eq boxes} that yields the solution $y^\diamond$ and let $\linearized^\diamond$ be the corresponding affine function defined in \eqref{eq underestimator}.
Let $y^*$ be an optimal solution to Problem~\ref{prob SMIQP}, so that $\objS(y^*) = \objSinf$, let $\C^*$ be a box among \eqref{eq boxes} such that $y^* \in \C^*$ and let $\linearized^*$ be the corresponding affine function.
We have
\begin{align}
\begin{split}
\label{eq claim good}
\objS(y^\diamond) 
& = \sum_{i=1}^k d_i (y_i^\diamond)^2 + \sum_{i=k+1}^n d_i (y_i^\diamond)^2 + h^\transp y^\diamond \\
& \le \linearized^\diamond(y_1^\diamond,\dots,y_k^\diamond) 
+ \sum_{i=k+1}^n d_i (y_i^\diamond)^2 + h^\transp y^\diamond 
+ \gamma_{\mathrm{neg}} k r^2 / \varphi^2 \\
& \le \linearized^*(y_1^*,\dots,y_k^*) 
+ \sum_{i=k+1}^n d_i (y_i^*)^2 + h^\transp y^*
+ \gamma_{\mathrm{neg}} k r^2 / \varphi^2 \\
& \le \objS(y^*) + \gamma_{\mathrm{neg}} k r^2 / \varphi^2.
\end{split}
\end{align}
The first inequality follows by applying the right inequality in \eqref{eq claim sub} to $\C^\diamond$ and $y^\diamond$.
The second inequality holds by definition of~$y^\diamond$: its left-hand side is the optimal value of the MICQP on $\C^\diamond$, which is smallest among all boxes, whereas its right-hand side is the value of the objective function of the MICQP on $\C^*$ at $y^*$, which is feasible for that MICQP since $y^*$ is feasible for Problem~\ref{prob SMIQP} and $y^* \in \C^*$.
The third inequality follows by applying the left inequality in \eqref{eq claim sub} to $\C^*$ and $y^*$.

We are now ready to show that $y^\diamond$ is an $\epsilon$-approximate solution to Problem~\ref{prob SMIQP}.
We have 
\begin{align*}
\frac{\objS(y^\diamond) - \objSinf}{\objSsup - \objSinf}
& \le \frac{\cancel{\gamma_{\mathrm{neg}}} k r^2}{\varphi^2} \cdot \frac{1}{\kappa \cancel{\gamma_{\mathrm{neg}}}} 
= \frac{k r^2}{\kappa \varphi^2} 
\le \epsilon.
\end{align*}
In the first inequality we used \eqref{eq claim good} and the assumption $\objSsup - \objSinf \ge \kappa \gamma_{\mathrm{neg}}$, and the last inequality holds by definition of~$\varphi$.

\smallskip
\noindent
\textbf{Running time.}
The algorithm solves $\varphi^k = \ceilL{r\sqrt{k/(\kappa\epsilon)}}^k$ MICQPs, one per box. Since $r=n^2$ and $k$ is fixed, this number is polynomial in $n$ and in $1/\epsilon$. Each of these MICQPs has size polynomial in the size of the instance and in $\size(\epsilon)$: $\varphi$ and the box endpoints $\ell_i,u_i$ have size $O(\log n+\size(\epsilon))$, and $\varphi$ is computable in polynomial time, being the ceiling of the square root of a rational number of size $O(\log n+\size(\epsilon))$; the remaining coefficients of Problem~\eqref{prob MICQP on box} are obtained from these and from the original data by $O(1)$ arithmetic operations. Each such MICQP has $p$ integer variables, and, by Lemma~\ref{lem MICQP}, is solved in polynomial time since $p$ is fixed. Hence the entire algorithm runs on a Turing machine in time polynomial in the size of the instance, in $\size(\epsilon)$, and in $1/\epsilon$, provided that $p$ and $k$ are fixed.
\end{prf}

\subsection{Theorems of the alternative}
Our next goal is to present two theorems of the alternative, one for Problem~\ref{prob SMIQP} and one for Problem~\ref{prob MIQP}.
In these results, we show that there is an algorithm that either finds an $\epsilon$-approximate solution, or finds a direction along which the instance splits into a number of sub-instances with one fewer integer variable.
Let $\S \subseteq \R^n$ be a bounded closed set.
Given a vector $\dir \in \R^n$, we define the \emph{width of~$\S$ along $\dir$} to be
\begin{align*}
\width_\dir (\S) := \max \bra{ \dir^\transp x : x \in \S } - \min \bra{ \dir^\transp x : x \in \S }.
\end{align*}
Now let $b^1, \dots , b^n$ be linearly independent vectors in $\R^n$ and consider a nonzero vector $\dir \in \R^n$ with $\dir^\transp b^1, \dots, \dir^\transp b^p$ integer and $\dir^\transp b^{p+1} = \cdots = \dir^\transp b^n = 0$.
Then $\dir^\transp x$ is an integer for every $x \in \Lambda_p(b^1, \dots , b^n)$.
Since $\dir^\transp x$ takes values in an interval of length $\width_\dir(\S)$ as $x$ ranges over $\S$, and takes only integer values on $\S \cap \Lambda_p(b^1,\dots,b^n)$, it follows that the number of hyperplanes orthogonal to $\dir$ that contain points of $\S \cap \Lambda_p(b^1, \dots , b^n)$ is at most $\floor{\width_\dir(\S)}+1$.
We will use the following three lemmas, all of which have already been used, at least implicitly, in previous algorithms of this type.
The first is a flatness-type result: it either finds a mixed integer point in a ball, or a direction along which the ball is flat.
\begin{lemma}[Lemma~5.2 in \cite{dP25MPA}]
\label{lem Lenstra width}
Let $a,c \in \Q^n$, let $r \in \Q$ with $r \ge 0$, let $b^1, \dots , b^n$ be linearly independent vectors in $\Q^n$, and let $p \in \{0,\dots,n\}$.
There is an algorithm that either finds a vector in $\B(a,r) \cap (\Lambda_p(b^1, \dots , b^n) + \{c\})$, 
or finds a nonzero vector $\dir \in \Q^n$ with $\dir^\transp b^1, \dots, \dir^\transp b^p$ integer and $\dir^\transp b^{p+1} = \cdots = \dir^\transp b^n = 0$ such that $\width_\dir(\B(a,r)) \le p \constlen$.
The algorithm runs on a Turing machine in time polynomial in $\size(a)+\size(c)+\size(r)+\sum_{i=1}^n\size(b^i)$.
\end{lemma}

\begin{lemma}[Lemma~4 in \cite{dP23bMPA} with $d=n$]
\label{lem plus minus}
Let $f: \R^n \to \R$ be a quadratic function of the form $f(y) = y^\transp D y + h^\transp y$, where $D$ is diagonal, and the entry of $D$ with the largest absolute value is the $1$st diagonal entry $d_1$. Let $y^+, y^- \in \R^n$ be such that $y_1^+ - y_1^- \ge 1$ and $\sum_{i=2}^n (y_i^+-y_i^-)^2 \le 1/4$. Let $\underline f$ and $\overline f$ be the minimum and maximum values attained by $f$ on the three vectors $y^+$, $y^-$, $(y^+ + y^-)/2$. Then $\overline f - \underline f \ge \frac{3}{16}\abs{d_1}$.
\end{lemma}

\begin{lemma}
\label{lem midpoint}
Let $b^1, \dots , b^n$ be linearly independent vectors in $\R^n$, let $p \in \{0,\dots,n\}$, and let $c \in \R^n$.
Let $y^+, y^- \in \Lambda_p(2b^1, \dots , 2b^n) + \{c\}$.
Then
$$
\frac{y^+ + y^-}{2} \in \Lambda_p(b^1, \dots , b^n) + \{c\}.
$$
\end{lemma}

\begin{prf}
Since $y^+, y^- \in \Lambda_p(2b^1, \dots , 2b^n) + \{c\}$, we can write
$$
y^+ = \sum_{i=1}^n \mu_i (2 b^i) + c, \qquad y^- = \sum_{i=1}^n \nu_i (2 b^i) + c,
$$
with $\mu_i, \nu_i \in \Z$ for $i=1,\dots,p$ and $\mu_i, \nu_i \in \R$ for $i=p+1,\dots,n$.
Averaging the two expressions, we obtain
$$
\frac{y^+ + y^-}{2} = \sum_{i=1}^n (\mu_i + \nu_i) b^i + c.
$$
Since $\mu_i + \nu_i \in \Z$ for $i=1,\dots,p$, the right-hand side is in $\Lambda_p(b^1, \dots , b^n) + \{c\}$.
\end{prf}

We are now ready to give our first theorem of the alternative.

\begin{proposition}[Theorem of the alternative for Problem~\ref{prob SMIQP}]
\label{prop alternative SMIQP}
Assume that, in Problem~\ref{prob SMIQP}, $D$ has at least one negative entry, and let $\epsilon\in(0,1]$ be rational.
There is an algorithm that either finds an $\epsilon$-approximate solution to Problem~\ref{prob SMIQP} and certifies that $\objSsup - \objSinf \ge 3 \gamma_{\max} /16$, where $\gamma_{\max}$ is the largest absolute value of an entry of $D$, or finds a nonzero vector $\dir \in \Q^n$  with $\dir^\transp b^1, \dots, \dir^\transp b^p$ integer and $\dir^\transp b^{p+1} = \cdots = \dir^\transp b^n = 0$ such that $\width_\dir(\B(0,n^{3/2})) \le 8 p \constlen n^{3/2}$.
The algorithm runs on a Turing machine in time polynomial in the size of the instance, in $\size(\epsilon)$, and in $1/\epsilon$, provided that $p$ and the number $k$ of negative entries of $D$ are fixed.
\end{proposition}

\begin{prf}
In this proof, we denote by $d_i$ the $i$th diagonal entry of $D$.
After possibly renaming the variables, we assume without loss of generality that the entry of $D$ with the largest absolute value is $d_1$, so that $\gamma_{\max} = \abs{d_1}$.

We denote by $e^1$ the first vector of the standard basis of $\R^n$, and we apply Lemma~\ref{lem Lenstra width} to $\B(3 e^1/4, 1/4)$ and the mixed integer lattice $\Lambda_p(2b^1,\dots,2b^n) + \{c\}$, which is contained in $\Lambda_p(b^1,\dots,b^n) + \{c\}$ since $2\mu \in \Z$ whenever $\mu \in \Z$.
Consider first the case where Lemma~\ref{lem Lenstra width} finds a nonzero vector $\dir' \in \Q^n$ with ${\dir'}^\transp (2b^1), \dots, {\dir'}^\transp (2b^p)$ integer and ${\dir'}^\transp (2b^{p+1}) = \cdots = {\dir'}^\transp (2b^n) = 0$ such that $\width_{\dir'}(\B(3 e^1/4,1/4)) \le p \constlen$.
We then set $\dir := 2{\dir'} \in \Q^n$ and we have $\dir^\transp b^1, \dots, \dir^\transp b^p$ integer and $\dir^\transp b^{p+1} = \cdots = \dir^\transp b^n = 0$.
Furthermore, we have
\begin{align*}
\width_\dir(\B(0,n^{3/2}))
& = 2 \width_{\dir'}(\B(0,n^{3/2})) = 8n^{3/2} \width_{\dir'}(\B(0,1/4)) \\
& = 8n^{3/2} \width_{\dir'}(\B(3 e^1/4,1/4)) \le 8 p \constlen n^{3/2}.
\end{align*}
Hence the statement of the proposition holds. 
Therefore, we now assume that Lemma~\ref{lem Lenstra width} finds a vector
$$
y^+ \in \B(3 e^1/4,1/4) \cap (\Lambda_p(2b^1, \dots , 2b^n) + \{c\}).
$$
Clearly, we also have $y^+ \in \B(0,1)$, which is contained in the polyhedron $\{y \in \R^n : Wy \le w\}$ by the spherical form; since $y^+$ also lies in $\Lambda_p(b^1,\dots,b^n) + \{c\}$, as observed above, it is feasible to Problem~\ref{prob SMIQP}.

Next,  we apply Lemma~\ref{lem Lenstra width} to $\B(-3 e^1/4, 1/4)$ and the mixed integer lattice $\Lambda_p(2b^1,\dots,2b^n) + \{c\}$.
Symmetrically, we can assume that Lemma~\ref{lem Lenstra width} finds a vector
$$
y^- \in \B(-3 e^1/4,1/4) \cap (\Lambda_p(2b^1, \dots , 2b^n) + \{c\}),
$$
which in particular implies that $y^-$ is in $\B(0,1)$, and so it is feasible to Problem~\ref{prob SMIQP}.

To conclude the proof, it suffices to show that $\objSsup - \objSinf \ge 3 \gamma_{\max}/16$.
Indeed, we can then use Proposition~\ref{prop SMIQP algorithm} (with $\kappa := 3/16$, the constant of Lemma~\ref{lem plus minus}) to find an $\epsilon$-approximate solution, since $\gamma_{\max} \ge \gamma_{\mathrm{neg}}$, where $\gamma_{\mathrm{neg}}$ is the largest absolute value among the negative entries of $D$, as in the statement of Proposition~\ref{prop SMIQP algorithm}.

Since $y^+ \in \B(3 e^1/4,1/4)$ and $y^- \in \B(-3 e^1/4,1/4)$, we obtain
$y^+_1 - y^-_1 \ge 1.$
The vectors $(y^+_2,\dots,y^+_n)$ and $(y^-_2,\dots,y^-_n)$ are both contained in the $(n-1)$-dimensional ball $\B(0,1/4)$, thus their distance is at most $1/2$.
In other words,
\begin{align*}
\sum_{i=2}^n (y^+_i - y^-_i)^2 
\le 1/4.
\end{align*}
We define the midpoint of the segment joining $y^+$ and $y^-$ as $y^\circ :=  (y^+ + y^-)/2.$
Note that also the vector $y^\circ$ is feasible to Problem~\ref{prob SMIQP}.
Indeed, it satisfies the linear inequalities by convexity, and it is in $\Lambda_p(b^1, \dots , b^n) + \{c\}$ by Lemma~\ref{lem midpoint}, since $y^+, y^- \in \Lambda_p(2b^1, \dots , 2b^n) + \{c\}$.

By Lemma~\ref{lem plus minus} (with $f:=\objS$), the minimum and maximum values attained by $\objS$ on $y^+,y^-,y^\circ$ differ by at least $\frac3{16}\abs{d_1}=\frac3{16}\gamma_{\max}$.
Since all three vectors are feasible to Problem~\ref{prob SMIQP}, we conclude that 
$\objSsup - \objSinf \ge 3 \gamma_{\max} /16.$

\smallskip
\noindent
\textbf{Running time.}
Each application of Lemma~\ref{lem Lenstra width} runs in time polynomial in the size of the instance, since the balls $\B(\pm 3e^1/4,1/4)$ have constant size and $b^1,\dots,b^n,c$ are part of the input. In the partition case, forming $\dir=2\dir'$ takes $O(n)$ further arithmetic operations. Otherwise, forming $y^\circ$ takes $O(n)$ arithmetic operations, and, by Proposition~\ref{prop SMIQP algorithm}, finding the $\epsilon$-approximate solution takes time polynomial in the size of the instance, in $\size(\epsilon)$, and in $1/\epsilon$ when $p$ and $k$ are fixed. Hence the entire algorithm runs on a Turing machine in time polynomial in the size of the instance, in $\size(\epsilon)$, and in $1/\epsilon$, provided that $p$ and $k$ are fixed.
\end{prf}

Next, we present our theorem of the alternative for Problem~\ref{prob MIQP}.

\begin{proposition}[Theorem of the alternative for Problem~\ref{prob MIQP}]
\label{prop alternative MIQP}
Assume that, in Problem~\ref{prob MIQP}, the matrix $H$ has at least one negative eigenvalue and $\{x \in \R^n: Wx \le w\}$ is full-dimensional and bounded, and let $\epsilon \in (0,1]$ be rational.
There is an algorithm that either finds an $\epsilon$-approximate solution to Problem~\ref{prob MIQP}, or finds a nonzero vector $\dir \in \Z^n$ with $\dir_{p+1} = \cdots = \dir_n = 0$ and a scalar $\rho \in \Q$ such that 
$$
\bra{\dir^\transp x : Wx \le w} \subseteq \sbra{\rho, \rho + 8 p \constlen n^{3/2}}.
$$
The algorithm runs on a Turing machine in time polynomial in the size of the instance, in $\size(\epsilon)$, and in $1/\epsilon$, provided that $p$ and the number $k$ of negative eigenvalues of $H$ are fixed.
\end{proposition}

\begin{prf}
We employ Proposition~\ref{prop reduction to SMIQP} with $\kappa := 3/16$, and we find
a nonsingular matrix $L \in \Q^{n \times n}$, a diagonal matrix $D \in \Q^{n \times n}$, and a vector $a \in \Q^n$ such that:
(i)
$H$ and $D$ have the same inertia, thus the number of negative eigenvalues of $H$ coincides with the number of negative entries of $D$.
(ii) The following problem is a spherical form MIQP, where $b^1,\dots,b^n$ denote the columns of $L^{-1}$:
\begin{align}
\label{prob from MIQP to SMIQP inproof}
\begin{split}
\min & \quad y^\transp D y + (h^\transp L + 2 a^\transp HL) y \\
\st & \quad WLy \le w - Wa \\
& \quad y \in \Lambda_p(b^1,\dots,b^n) - \{L^{-1} a\}.
\end{split}
\end{align}
(iii)
If $y^\diamond$ is an $\epsilon/2$-approximate solution to Problem~\eqref{prob from MIQP to SMIQP inproof}
and $f_{\sup} - f_{\inf} \ge 3 \gamma_{\min}/16$ (where $f$ is the objective function of Problem~\eqref{prob from MIQP to SMIQP inproof}, $f_{\inf}$, $f_{\sup}$ are the infimum and supremum of $f$ on the feasible region of Problem~\eqref{prob from MIQP to SMIQP inproof}, and $\gamma_{\min}$ is the smallest absolute value among the nonzero entries of $D$), then $x^\diamond := L y^\diamond + a$ is an $\epsilon$-approximate solution to Problem~\ref{prob MIQP}.

We use Proposition~\ref{prop alternative SMIQP} (with $\epsilon := \epsilon/2$) and we either find an $\epsilon/2$-approximate solution $y^\diamond$ to Problem~\eqref{prob from MIQP to SMIQP inproof} and we know that $f_{\sup} - f_{\inf} \ge 3 \gamma_{\max} /16$, where $\gamma_{\max}$ is the largest absolute value of an entry of $D$, or we find a nonzero vector $\dir' \in \Q^n$  with ${\dir'}^\transp b^1, \dots, {\dir'}^\transp b^p$ integer and ${\dir'}^\transp b^{p+1} = \cdots = {\dir'}^\transp b^n = 0$ such that $\width_{\dir'}(\B(0,n^{3/2})) \le 8 p \constlen n^{3/2}$.

In the first case, since $\gamma_{\max} \ge \gamma_{\min}$, we have $f_{\sup} - f_{\inf} \ge 3 \gamma_{\min}/16$, and hence $x^\diamond := L y^\diamond + a$ is an $\epsilon$-approximate solution to Problem~\ref{prob MIQP} by (iii), and we are done.
Consider now the second case.
Let $\dir := {L^{-\transp}} {\dir'} \in \Q^n$.
Since $L$ is nonsingular, $\dir$ is nonzero.
For $j=1,\dots,n$, we have $\dir_j = {b^j}^\transp {\dir'}$.
Thus, ${\dir'}^\transp b^1, \dots, {\dir'}^\transp b^p$ integer implies $\dir_1, \cdots, \dir_p$ integer, and ${\dir'}^\transp b^{p+1} = \cdots = {\dir'}^\transp b^n = 0$ implies $\dir_{p+1} = \cdots = \dir_n = 0$.
Furthermore,
\begin{align*}
\width_{\dir} & \bra{x \in \R^n : Wx \le w}
 = \max \bra{ {\dir}^\transp x : Wx \le w } - \min \bra{ {\dir}^\transp x : Wx \le w } \\
& = \max \bra{ {\dir'}^\transp L^{-1} x : Wx \le w } - \min \bra{ {\dir'}^\transp L^{-1} x : Wx \le w } \\
& = \max \bra{ {\dir'}^\transp (y + L^{-1} a) : WLy \le w - Wa } - \min \bra{ {\dir'}^\transp (y + L^{-1} a) : WLy \le w - Wa } \\
& = \max \bra{ {\dir'}^\transp y : WLy \le w - Wa } - \min \bra{ {\dir'}^\transp y : WLy \le w - Wa } \\
& = \width_{\dir'} \bra{y \in \R^n : WLy \le w - Wa} \\
& \le \width_{\dir'} \B(0,n^{3/2}) \\
& \le 8 p \constlen n^{3/2}.
\end{align*}

Let $\rho := \min\{{\dir}^\transp x : Wx \le w\}$, and note that we can calculate it in polynomial time by solving a linear programming problem.
We obtain
\begin{align*}
\bra{\dir^\transp x : Wx \le w} \subseteq \sbra{\rho, \rho + 8 p \constlen n^{3/2}}.
\end{align*}
This concludes the second case.

\smallskip
\noindent
\textbf{Running time.}
By Proposition~\ref{prop reduction to SMIQP}, computing $L,D,a$ takes time polynomial in the size of the instance and in $\size(\epsilon)$; by (i), $D$ has $k$ negative entries, so Proposition~\ref{prop alternative SMIQP} (applied with $\epsilon/2$) runs in time polynomial in the size of the instance, in $\size(\epsilon)$, and in $1/\epsilon$ when $p$ and $k$ are fixed. In the first case, forming $x^\diamond=Ly^\diamond+a$ takes $O(n^2)$ further arithmetic operations; in the second, forming $\dir=L^{-\transp}\dir'$ takes $O(n^2)$ arithmetic operations and computing $\rho$ takes polynomial time via a linear program. Hence the entire algorithm runs on a Turing machine in time polynomial in the size of the instance, in $\size(\epsilon)$, and in $1/\epsilon$, provided that $p$ and $k$ are fixed.
\end{prf}

\subsection{Auxiliary results}

The proof of Theorem~\ref{th MIQP algorithm} relies on a recursive application of Proposition~\ref{prop alternative MIQP}.
Before turning to it, we state a few additional results that we will use.

\begin{lemma}[Theorem 4 in \cite{dPDeyMol17MPA}]
\label{lem NP}
Assume that Problem~\ref{prob MIQP} is feasible and the objective function is bounded below on the feasible region.
There is an optimal solution to Problem~\ref{prob MIQP} of size bounded by an integer $\psi$, which is polynomial in the size of the instance. 
\end{lemma}

We remark that, even though Theorem~4 in~\cite{dPDeyMol17MPA} does not give $\psi$ explicitly, a formula for $\psi$, as a function of the size of the instance, can be derived from its proof.

\begin{lemma}[Theorem 1.1 in \cite{dP25SIOPT}]
\label{lem full dim}
Let $W\in\Q^{m\times n}$, $w\in\Q^m$, $p\in\{0,1,\dots,n\}$, and consider the sets
$$
\P := \{x\in\R^n : Wx\le w\}, \qquad \S := \P\cap(\Z^p\times\R^{n-p}).
$$
There is a polynomial-time algorithm that either returns that $\S$ is empty, or finds $p'\in\{0,1,\dots,p\}$, $n'\in\{p',p'+1,\dots,p'+n-p\}$, $\bar x\in\Z^p\times\Q^{n-p}$, and $T\in\Q^{n\times n'}$ of full rank such that, defining $\tau:\R^{n'}\to\R^n$ by $\tau(x'):=\bar x+Tx'$, the polyhedron
$$
\P' := \{x'\in\R^{n'} : WTx'\le w-W\bar x\}
$$
is full-dimensional, and
$$
\P = \tau(\P'), \qquad \S = \tau\pare{\P'\cap(\Z^{p'}\times\R^{n'-p'})}.
$$
Furthermore, if an equality $\dir^\transp x=\beta$ with $\dir_{p+1}=\cdots=\dir_n=0$ is valid for $\P$, then $p'\le p-1$.
\end{lemma}

\begin{lemma}[Theorem 3 in \cite{dP23bMPA}]
\label{lem sym decomp}
Let $H \in \Q^{n \times n}$ be symmetric. There is an algorithm that computes a nonsingular matrix $L\in\Q^{n\times n}$ and a diagonal matrix $D\in\Q^{n\times n}$ such that $L^\transp HL=D$. The algorithm is strongly polynomial: it performs a number of arithmetic operations polynomial in $n$, and runs on a Turing machine in time polynomial in $\size(H)$.
\end{lemma}

\begin{lemma}[Lemma 7.3 in \cite{dP25MPA}]
\label{lem partition}
Let $\S^1,\dots,\S^t \subseteq \R^n$.
Let $f : \bigcup_{i=1}^t \S^i \to \R$ and assume that $f$ has a minimum on $\bigcup_{i=1}^t \S^i$.
Let $\epsilon \in [0,1]$.
For $i=1,\dots,t$, let $x^i$ be an $\epsilon$-approximate solution to the optimization problem $\inf\{f(x) : x \in \S^i\}$.
Then, each optimal solution to $\min\{f(x) : x \in \{x^1,\dots,x^t\}\}$ is an $\epsilon$-approximate solution to the optimization problem $\inf\{f(x) : x \in \bigcup_{i=1}^t \S^i\}$.
\end{lemma}

\subsection{Proof of Theorem~\ref{th MIQP algorithm}}

\begin{prf}
We now present our approximation algorithm for Problem~\ref{prob MIQP}.
The input to the algorithm is an instance of Problem~\ref{prob MIQP}, that is, specific data $H,h,W,w$, such that the objective function is bounded below on the feasible region.
We denote by $k$ the number of negative eigenvalues of the matrix $H$ of the input instance, that is, $k = k_-$ in the statement of Theorem~\ref{th MIQP algorithm}.

\smallskip
\noindent
\textbf{Initialization step 1: Bounded polyhedron.}
By Lemma~\ref{lem NP} (together with the remark following it), there is a formula, depending only on the size of the instance, for an integer $\psi$ such that: if the instance is feasible, then it has an optimal solution of size at most $\psi$.
We compute $\psi$ from this formula, regardless of whether the instance is feasible, and construct a new instance, obtained from the original one by adding the linear constraints $-2^\psi \le x_i \le 2^\psi$, for $i=1,\dots, n$.
The polyhedron corresponding to this new instance is 
$\bra{x \in \R^n : Wx \le w, \ -2^\psi \le x_i \le 2^\psi, \text{ for } i=1,\dots, n}$, which is bounded.
Furthermore, the size of the new instance is polynomial in the size of the former.
If the original instance is infeasible, then the new instance remains infeasible, since its feasible region is contained in the one of the original instance.
If the original instance is feasible, then, since some optimal solution has size at most $\psi$ and hence satisfies the added inequalities, the feasible region remains nonempty and the infimum of the objective function is unchanged; since the supremum can only decrease, an $\epsilon$-approximate solution to the new instance is also an $\epsilon$-approximate solution to the original one.
Therefore, we may replace the original instance by this new instance, whose corresponding polyhedron is bounded, without affecting feasibility, and so that finding an $\epsilon$-approximate solution to the new instance (or detecting its infeasibility) suffices.
For notational simplicity, we now use $H,h,W,w$ to denote the data of the new instance, so we now assume that the system $Wx \le w$ contains the linear inequalities $-2^\psi \le x_i \le 2^\psi$, for $i=1,\dots, n$.

\smallskip
\noindent
\textbf{Initialization step 2: Full dimensional polyhedron.}
We now apply Lemma~\ref{lem full dim} to the feasible region of the instance.
If we find out that the feasible region is empty, then the instance is infeasible, and we are done.
Otherwise, we obtain $p' \in \bra{0,1,\dots,p}$, $n' \in \bra{p',p'+1,\dots,p'+n-p}$, 
$\bar x \in \Z^p \times \Q^{n-p}$, and $T \in \Q^{n \times n'}$ of full rank, as in the statement of Lemma~\ref{lem full dim}.
We perform the change of variables $x = \bar x + T x'$ to the instance and, after dropping the constant $\bar x^\transp H \bar x + h^\transp \bar x$ in the objective function, we obtain the instance
\begin{align}
\label{prob MIQP fulldim}
\begin{split}
\min & \quad {x'}^\transp H' x' + {h'}^\transp x' \\
\st & \quad W'x' \le w' \\
& \quad x' \in \Z^{p'} \times \R^{n'-p'},
\end{split}
\end{align} 
where $H' := T^\transp H T$,
${h'}^\transp := h^\transp T + 2 \bar x^\transp HT$,
$W' := WT$, 
and $w' := w - W \bar x$.
By Lemma~\ref{lem full dim}, the polyhedron $\bra{x' \in \R^{n'} : W'x' \le w'}$ is full-dimensional; furthermore, since $\bra{x \in \R^n : Wx \le w}$ is bounded and $\tau(x'):=\bar x+Tx'$ maps this polyhedron bijectively onto $\bra{x \in \R^n : Wx \le w}$, the polyhedron $\bra{x' \in \R^{n'} : W'x' \le w'}$ is bounded as well.
The number of negative eigenvalues of $H'$ is at most the number of negative eigenvalues of $H$.
Indeed, for a symmetric matrix, the number of negative eigenvalues equals the largest dimension of a subspace on which the associated quadratic form is negative definite; and if ${x'}^\transp H' x'$ is negative definite on a subspace $V' \subseteq \R^{n'}$, then, since $T$ has full column rank, $x^\transp H x$ is negative definite on the subspace $T V' \subseteq \R^n$, which has the same dimension as $V'$.
Since the definition of $\epsilon$-approximate solution is preserved under changes of variables and translations of the objective function, and since the instance is feasible if and only if Instance~\eqref{prob MIQP fulldim} is, it then suffices to find an $\epsilon$-approximate solution to Instance~\eqref{prob MIQP fulldim}, or to detect that it is infeasible.

\smallskip
\noindent
\textbf{Recursive~steps.}
The remainder of the algorithm will be applied recursively to several instances of Problem~\ref{prob MIQP} whose corresponding polyhedron is full-dimensional and bounded.
In the first iteration, the instance considered is Instance~\eqref{prob MIQP fulldim}.
For notational simplicity, we assume that the instance considered has again data $H,h,W,w$.
Therefore, we assume that the polyhedron $\{x\in\R^n:Wx\le w\}$ is full-dimensional and bounded.

\smallskip
\noindent
\textbf{Recursive~step~1: Convex case.}
By Lemma~\ref{lem sym decomp}, there is a strongly polynomial algorithm that finds a nonsingular matrix $L\in\Q^{n\times n}$ and a diagonal matrix $D\in\Q^{n\times n}$ such that $L^\transp HL=D$; we use this only to check whether $H$ is positive semidefinite, and $L,D$ play no further role below. By Sylvester's law of inertia, $H$ and $D$ have the same inertia; in particular, $H$ is positive semidefinite if and only if every diagonal entry of $D$ is nonnegative, which can be checked in polynomial time.
If $H$ is positive semidefinite, then, since $\{x\in\R^n:Wx\le w\}$ is bounded, we apply Lemma~\ref{lem MICQP} to our instance.
If Lemma~\ref{lem MICQP} finds that the instance is infeasible, we are done with this instance.
Otherwise, Lemma~\ref{lem MICQP} finds an optimal solution $x^*$ to our instance in polynomial time, which is also an $\epsilon$-approximate solution.
Thus, in the remainder of the proof, we assume that $H$ is not positive semidefinite; in particular, since $H$ and $D$ have the same inertia, the matrix $H$ of the instance considered has at least one negative eigenvalue.

\smallskip
\noindent
\textbf{Recursive~step~2: Approximation or partition.}
We now apply the algorithm in Proposition~\ref{prop alternative MIQP}.
If the algorithm finds an $\epsilon$-approximate solution to our instance, then we are done with this instance.
Otherwise, the algorithm finds a nonzero vector $\dir \in \Z^n$ with $\dir_{p+1} = \cdots = \dir_n = 0$ and a scalar $\rho \in \Q$ such that 
$$
\bra{\dir^\transp x : Wx \le w} \subseteq \sbra{\rho, \rho + 8 p \constlen n^{3/2}}.
$$
Note that $p \ge 1$: indeed, $\dir$ is nonzero and satisfies $\dir_{p+1} = \cdots = \dir_n = 0$, which is impossible if $p=0$.
Thus the partition case can occur only when $p \ge 1$, so that each partition strictly decreases the number of integer variables; this is what guarantees that the recursion terminates.
Note also that $\dir^\transp x$ is an integer for every $x \in \Z^p \times \R^{n-p}$, since $\dir \in \Z^n$ and $\dir_{p+1} = \cdots = \dir_n = 0$.
Therefore, every feasible point of our instance is contained in one of the hyperplanes $\{x \in \R^n : {\dir}^\transp x = \beta\}$, where $\beta$ ranges over the integers with $\ceil{\rho} \le \beta \le \ceil{\rho + 8 p \constlen n^{3/2}}$.
For each such $\beta$, we define the instance
\begin{align}
\label{prob MIQP beta}
\begin{split}
\min & \quad x^\transp H x + h^\transp x \\
\st & \quad Wx \le w \\
& \quad \dir^\transp x = \beta \\
& \quad x \in \Z^p \times \R^{n-p}.
\end{split}
\end{align}
For each $\beta$, we apply Lemma~\ref{lem full dim} to the feasible region of Instance~\eqref{prob MIQP beta}, as in Step~2.
If we find out that the feasible region is empty, we do not need to consider further Instance~\eqref{prob MIQP beta} for such value of $\beta$.
Otherwise, we transform, with a change of variables, Instance~\eqref{prob MIQP beta} into a new instance defined over a full-dimensional polyhedron, which we denote by Instance~\eqref{prob MIQP beta}'.
This polyhedron is also bounded: the feasible region of Instance~\eqref{prob MIQP beta} is contained in the bounded set $\{x \in \R^n : Wx \le w\}$, and, by Lemma~\ref{lem full dim}, it is the image of the polyhedron of Instance~\eqref{prob MIQP beta}' under an injective affine map.
Since $\dir_{p+1} = \cdots = \dir_n = 0$, Instance~\eqref{prob MIQP beta}' has at most $p-1$ integer variables and at most $n-p$ continuous variables.
Moreover, exactly as in Initialization step~2, the number of negative eigenvalues of the Hessian of Instance~\eqref{prob MIQP beta}' is at most that of $H$.

\smallskip
\noindent
\textbf{Recursive~step~3: Recursion.}
We apply the recursive steps of the algorithm to each obtained Instance~\eqref{prob MIQP beta}'.
Since each partition in Recursive~step~2 decreases the number of integer variables by at least one, the recursion terminates.

\smallskip
\noindent
\textbf{Termination.}
Each $\epsilon$-approximate solution of a sub-instance found by the algorithm is mapped back to the corresponding feasible solution to the original instance by inverting the affine transformations performed by the algorithm (via Lemma~\ref{lem full dim}) to construct such sub-instance.
The algorithm then returns one of these feasible solutions with minimum objective value.
Lemma~\ref{lem partition} implies that the returned solution is an $\epsilon$-approximate solution to the original instance.
In case the algorithm never found any $\epsilon$-approximate solutions of any sub-instance, the algorithm returns that the original instance is infeasible.

\smallskip
\noindent
\textbf{Running time.}
The initialization steps run in polynomial time: computing $\psi$ and adding the box constraints (Initialization step~1), and applying Lemma~\ref{lem full dim} (Initialization step~2). For the recursion, since each partition decreases the number of integer variables by at least one, its depth is at most $p$; and since each partition creates at most $8p\constlen n^{3/2}+1$ sub-instances, the total number of iterations is at most $(p+1)\pare{8p\constlen n^{3/2}+1}^p$, which is polynomial in $n$ because $p$ is fixed. As this depth is the constant $p$, every instance considered has size polynomial in the size of the original instance, and, by Recursive step~2, its number of negative eigenvalues never exceeds $k$. Each iteration checks positive semidefiniteness via Lemma~\ref{lem sym decomp}, and then either solves a MICQP via Lemma~\ref{lem MICQP} or invokes Proposition~\ref{prop alternative MIQP} and applies Lemma~\ref{lem full dim} to each sub-instance; when $p$ and $k$ are fixed, each of these runs in time polynomial in the size of the instance, in $\size(\epsilon)$, and in $1/\epsilon$. Hence the entire algorithm runs on a Turing machine in time polynomial in the size of the instance, in $\size(\epsilon)$, and in $1/\epsilon$, provided that $p$ and $k$ are fixed.
\end{prf}

%
%
%
%

\bigskip

\begin{small}
\noindent
\textbf{Funding: }A. Del Pia is partially supported by AFOSR Grant FA9550-23-1-0433 and ONR Grant N00014-25-1-2490. Any opinions, findings, conclusions, or recommendations expressed in this material are those of the authors and do not necessarily reflect the views of the Air Force Office of Scientific Research or the Office of Naval Research.
\end{small}

\ifthenelse {\boolean{MPA}}
{
\bibliographystyle{spmpsci}
}
{
\bibliographystyle{plain}

}

\end{document}